\documentclass[12pt,fleqn]{amsart}
\usepackage{amsmath,amssymb,amsfonts,amsthm}
\usepackage{appendix}
\usepackage{mathtools}
\usepackage[hidelinks]{hyperref}
\hypersetup{colorlinks=true,linkcolor=blue, citecolor=blue, urlcolor=blue}
\usepackage{tikz}
\usepackage{tikz-cd}
\usepackage{enumerate}
\usepackage{graphicx, array}
\usepackage{float}
\usepackage{xfrac} 
\usepackage{arydshln}
\usepackage{dsfont}
\usepackage{comment}
\usepackage{multirow}
\usepackage{bigstrut}

\usepackage{graphicx}
\graphicspath{{figures/}}
\usepackage[
labelfont=small,
labelformat=parens,
subrefformat=parens
]{subcaption}

\theoremstyle{plain}
\newtheorem{theorem}{Theorem}[section]
\newtheorem{lemma}[theorem]{Lemma}

\newtheorem{proposition}[theorem]{Proposition}
\theoremstyle{definition}
\newtheorem{definition}[theorem]{Definition}

\theoremstyle{remark}
\newtheorem{remark}[theorem]{Remark}

\usepackage[T5,T1]{fontenc}
\newcommand{\vietnamese}[1]{{\fontencoding{T5}\selectfont#1}}

\usepackage[
backend=biber,
maxalphanames=4,
style=alphabetic,
sorting=nyt
]{biblatex}
\DeclareFieldFormat*{title}{\textit{#1}}
\DeclareFieldFormat{journaltitle}{\textnormal{#1}}
\DeclareListFormat{publisher}{\textnormal{#1}}

\usepackage[left=2.5cm, right=2.5cm, top=3cm, bottom=3cm]{geometry}

\newcommand{\Z}{\mathbb{Z}}
\newcommand{\R}{\mathbb{R}}
\newcommand{\C}{\mathbb{C}}

\DeclareMathOperator{\ad}{ad}

\DeclareMathOperator{\disc}{disc}

\newcommand{\abs}[1]{\left\vert{#1}\right\vert}

\title{Double flip bifurcations in $\Z/2\Z$-symmetric Hamiltonian systems}

\author{Konstantinos Efstathiou}
\address{Konstantinos Efstathiou:
  Zu Chongzhi Center and Division of Natural and Applied Sciences \\
  Duke Kunshan University\\
  8 Duke Avenue,
  215316 Kunshan, Jiangsu, China}
\email{k.efstathiou@dukekunshan.edu.cn}

\author{Tobias V\r{a}ge Henriksen}
\address{Tobias V\r{a}ge Henriksen: 
  Bernoulli Institute for Mathematics,
  Computer Science and Artificial Intelligence\\
  University of Groningen\\ 
  P.O. Box 407, 9700 AK Groningen, The Netherlands,
  \textnormal{and}
  Department of Mathematics\\ 
  University of Antwerp\\ 
  Middelheimlaan 1, B-2020 Antwerp, Belgium
}
\email{t.v.henriksen@rug.nl}

\author{Sonja Hohloch}
\address{Sonja Hohloch: 
  Department of Mathematics\\ 
  University of Antwerp\\ 
  Middelheimlaan 1, B-2020 Antwerp, Belgium
}
\email{sonja.hohloch@uantwerpen.be}

\begin{document}

\date{\today}

\begin{abstract}
In this paper we introduce a new bifurcation in Hamiltonian systems, which we call the double flip bifurcation. The Hamiltonian depends on two parameters, one of which controls the double flip bifurcation. The result of the bifurcation is the occurrence of two Hamiltonian flip bifurcations with respect to the other parameter. The two Hamiltonian flip bifurcations are simultaneous with respect to the first parameter, and are connected by a curve-segment of singular points. We find a normal form for Hamiltonians describing systems going through double flip bifurcations, and compute said normal form for some examples.
\end{abstract}
\maketitle

\section{Introduction}

The theory of integrable (Hamiltonian) systems lies at the intersection of the theory of dynamical systems, symplectic geometry, ODEs, Lie theory, and classical mechanics. An integrable system is a triple $(M,\omega,F=(f_{1},\dots,f_{n}))$, where $(M,\omega)$ is a symplectic manifold of dimension $2n$, and $F : M \to \R^{n}$ is a smooth mapping, called the momentum map. The rank of $F$ is almost everywhere maximal, and the Poisson bracket, induced by $\omega$, vanishes for every pair $f_{i},f_{j}$, where $i,j \in \{1,\dots,n\}$. When $F$ has $n$ components we say that the integrable system has $n$ degrees of freedom, where each $f_{i}$ is called an integral of motion. 

Three important directions in the study of integrable systems (or, more generally, dynamical systems) are as follows:
\begin{enumerate}
    \item search for new phenomena, or
    \item search for local classifications and normal forms, or
    \item search for global classifications.
\end{enumerate}
The first direction naturally gives rise to the second, which often is necessary for the third direction, which again may lead to one of the first two directions. This work was motivated by the first direction, when we came across peculiar behaviour in integrable systems on Hirzebruch surfaces (see Section \ref{sec:ex_Hirzebruch}), which we decided to study locally. More precisely, we study a new bifurcation phenomena, which we describe by a local normal form, and which we compute for some examples. Concerning the third direction, a description of this bifurcation will play an essential role, as it produces so-called \textit{hyperbolic semitoric systems} (see below), for which attempts for a global classification are ongoing.

Now, let us put our results in context. Hyperbolic semitoric system are generalisiations of a particularly simple class of integrable systems, called toric. In toric systems, every integral of motion induces an $S^{1}$-action. This, in particular, restricts the possible components of singular points to be elliptic, or elliptic combined with regular components. Furthermore, the image of the momentum map is a polygon, which Delzant \cite{Delzant1988} showed to classify toric systems up to symplectomorphism.

Many systems in nature do not fall into the class of toric systems, but are examples of what we call semitoric systems. These are systems with $2$ degrees of freedom, for which one of the integrals of motion induces an $S^{1}$-action, while the other is arbitrary. Semitoric systems were classified by Pelayo and \vietnamese{{V{\~u} Ng{\d o}c}} \cite{Pelayo2009, Pelayo2011} by five invariants, one of them generalising the polygon invariant for toric systems. 

In addition to the regular and elliptic components of singularities in toric systems, the singularities of a semitoric system can also have focus-focus components. A focus-focus singularity may appear in a system as an elliptic-elliptic singularity goes through a certain bifurcation. For example, if the elliptic-elliptic singularity lives in a toric system, it may go through a Hamiltonian Hopf bifurcation (see van der Meer \cite{vanderMeer1985}), and it becomes focus-focus. Examples of systems going through Hamiltonian Hopf bifurcations include the Lagrange top \cite{Cushman1990}, the $3$D H{\'e}non-Heiles family \cite{Hanssmann2002}, the Hydrogen atom in crossed fields \cite{Efstathiou2004}, and Gaudin models \cite{Henriksen2025}. For more examples, see van der Meer's notes \cite[Section 5.4]{vanderMeer2017}. 

There is still a large restriction on which singularities can appear in semitoric systems. Firstly, they are all non-degenerate, but not even all non-degenerate types of singularities are included. Hyperbolic semitoric systems were introduced by Dullin and Pelayo \cite{DullinPelayo2016} to account for this. These are similar to semitoric systems, but do also allow singularities to have hyperbolic components. They also allow isolated degenerate singularities. Hyperbolic semitoric systems are not yet classified, but a first step was recently made by the first and third author together with Santos \cite{Efstathiou2024}, where they introduced an affine invariant akin to the polygon invariant for the toric and semitoric systems.

Singularities with hyperbolic components, as well as isolated degenerate singularities, can appear as a result of a subcritical Hamiltonian Hopf bifurcation, for which the degenerate singularities are cusps (also called parabolic). Note that this is a bifurcation which appears frequently in $2$ degree of freedom systems in which one of the integrals of motion, in suitable local coordinates $\{q_{1},p_{1},q_{2},p_{2}\}$, takes the form
\begin{equation} \label{eq:resonance}
\frac{m_{1}}{2}(q_{1}^{2}+p_{1}^{2}) + \frac{m_{2}}{2}(q_{2}^{2}+p_{2}^{2}),
\end{equation}
with $m_{1} = 1$ and $m_{2} = -1$. Such systems are said to be in $(1:-1)$-resonance. The fibre of a hyperbolic singularity arising this way is a stacked torus, see Figure \ref{fig:stacked-torus}. 

\begin{figure}
\def\scale{0.3}
\centering
\begin{subfigure}{\scale\linewidth}
    \centering
    \includegraphics[width=0.9\linewidth]{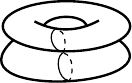}
    \caption{A stacked torus, also known as bi-torus or double torus.}
    \label{fig:stacked-torus}
\end{subfigure}
\begin{subfigure}{\scale\linewidth}
    \centering
    \includegraphics[width=0.9\linewidth]{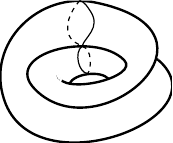}
    \caption{A curled torus.}
    \label{fig:curled-torus}
\end{subfigure}
\caption{Two different fibres of a hyperbolic singularity.}
\label{fig:tori}
\end{figure}

In the current paper, we present a novel type of bifurcation producing hyperbolic singularities and isolated degenerate singularities. This bifurcation appears in $(1:-2)$-resonant systems, i.e.\ with $m_{1} = 1$ and $m_{2} = -2$ in \eqref{eq:resonance}. Periodic orbits exposed to the $(1:-2)$-resonance has half period. For instance, this leads to the fibre of a hyperbolic singularity being a curled torus, see Figure \ref{fig:curled-torus}. Furthermore, in this resonance, Kalashnikov \cite{Kalashnikov1998} conjectures that there are two generic degenerate singularities. These singularities correspond to Hamiltonian flip bifurcations, which in $2$-degree of freedom systems are period-doubling bifurcations. This is explained in more detail below. See Cushman, Dullin, Han{\ss}mann and Schmidt \cite{Cushman2007} for more information on this resonance.


Now, to make the results of this paper more precise, consider the reduction of an integrable system to a $1$-degree of freedom system, where the integral of motion is the Hamiltonian $H_{j,t}$. The subscript $j,t$ indicates that the Hamiltonian depends on two parameters $j$ and $t$. We show that, if (the Taylor expansion of) $H_{j,t}$ satisfies certain properties, the system goes through a bifurcation not previously encountered in the literature, which we call the \emph{double flip bifurcation}. Below we give a description of the bifurcation, but let us first say a few words about the parameters $j$ and $t$. 

The double flip bifurcation happens as one of the parameters, say $t$, is varied. The other parameter can be thought of as arising from the reduction of an $n$-degree of freedom system, with $n \geq 2$. That is, if there is another integral of motion, say $J$, then let the parameter $j$ trace the energy-levels of $J$, namely let $J = j$. When we want to illustrate our results by figures, we will usually consider $n=2$, which can be easily visualised. That is, we may draw the bifurcation diagram, i.e.\ the image of the critical points under the momentum map, in $2$ dimensions, where $J$ and $H_{j,t}$ are the coordinate axes of the diagram. The word ``double'' in the name double flip bifurcation comes from the fact that the bifurcation, with respect to $t$, produces two bifurcations in the bifurcation diagram, with respect to the parameter $j$.

Let us now explain the bifurcations with respect to $j$. They are examples of the so-called \emph{Hamiltonian flip bifurcation}, see Han{\ss}mann \cite[pp.\ 33-34]{Hanssmann2006}. In $2$-degree of freedom systems, this is a period-doubling bifurcation. That is, a stable singularity in the sense of Lyapunov goes through a Hamiltonian flip bifurcation if it loses its stability, and simultaneously produces an additional stable singularity. If the singularity was instead unstable and gains stability while an additional unstable singularity is produced, then we talk of a \emph{dual Hamiltonian flip bifurcation}. The two Hamiltonian flip bifurcations resulting from the double flip bifurcation can be of any combination of the dual and non-dual variants. That is, they may both be non-dual, both dual, or one of each.

We now recall local models for the two variants of the Hamiltonian flip bifurcation. Note that they are similar to the pitchfork bifurcation up to a $\Z_{2} := \Z/2\Z$ isotropy group. A local model for the Hamiltonian flip bifurcation is given by Bolsinov and Fomenko \cite[Section 10.6, Example 2]{Bolsinov2004}: let $M = \R^{4}$ with coordinates $q_{1},p_{1},q_{2},p_{2}$ and canonical symplectic structure $\omega = dq_{1} \wedge dp_{1} + dq_{2} \wedge dp_{2}$. Let $F = (J,H_{1})$ be given by $J = p_{2}$ and $H_{1} = p_{1}^{2} + q_{1}^{4} - q_{1}^{2}p_{2}$. For some $j \in \R$, consider the energy-level $J = j$. Then we reduce $H_{1}$ by the action induced by $J$, obtaining
\begin{equation*}
H_{1,j}^{\textup{red}}(q_{1},p_{1}) := p_{1}^{2} + q_{1}^{4} - q_{1}^{2}j.
\end{equation*}
In Figure \ref{fig:type-H1} the change in stability of a singularity going through a Hamiltonian flip bifurcation is drawn along increasing $j$. It changes from a centre (stable) to a saddle (unstable), with two additional centres emerging. The two new centres are identified under the action by the $\Z_{2}$ isotropy group. The image of the momentum map is drawn in Figure \ref{fig:bif-H1}, where the black curves are the bifurcation diagram.

\begin{figure}[t]
\def\scale{0.75}
\begin{subfigure}{0.45\linewidth}
    \centering
    \includegraphics[width=0.9\linewidth]{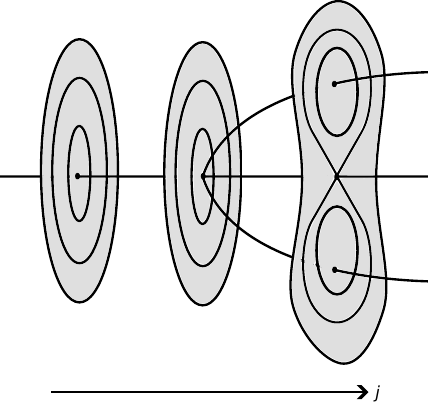}
    \caption{The types of singularities of the $H_{1,j}^{\textup{red}}$. Note that the upper and lower half are identified by the $\Z_{2}$-isotropy group.}
    \label{fig:type-H1}
\end{subfigure}
\begin{subfigure}{0.45\linewidth}
    \centering
    \includegraphics[width=0.9\linewidth]{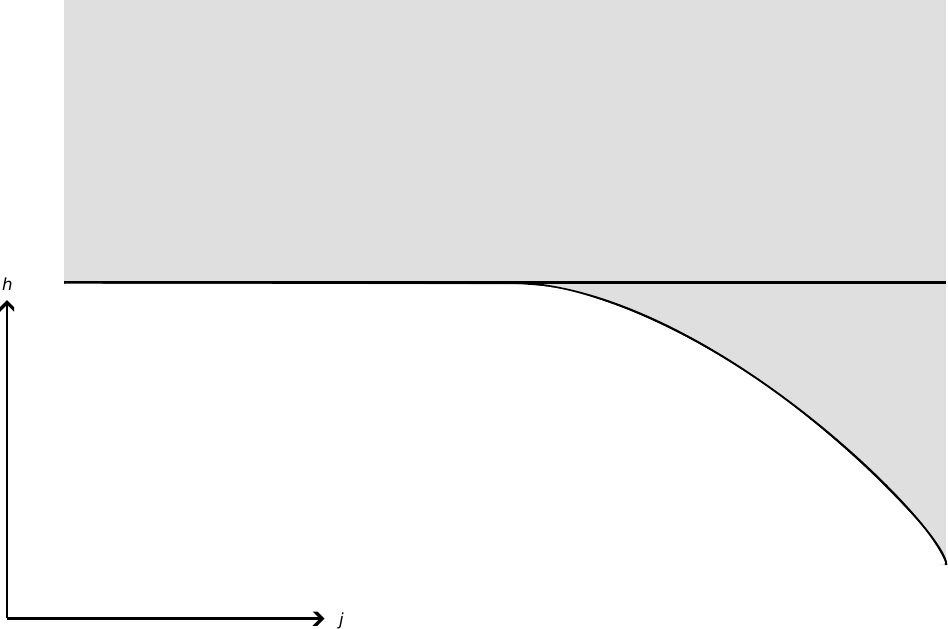}
    \caption{The black line and curve are the bifurcation diagram, while the grey area consist of regular values.}
    \label{fig:bif-H1}
\end{subfigure}
\caption{}
\label{fig:H1}
\end{figure}

The branching point in Figure \ref{fig:H1} is a degenerate singularity. In Kalashnikov's \cite{Kalashnikov1998} list of generic degenerate singularities in $2$-degree of freedom integrable systems, this is the $2_{a}$-singularity. 

The dual Hamiltonian flip bifurcation is also discussed by Bolsinov and Fomenko \cite[Section 10.6, Example 3]{Bolsinov2004}. This has a local model given by $H_{2} = p_{1}^{2} - q_{1}^{4} + q_{1}^{2}p_{2}$ and $J = p_{2}$, and so the reduced Hamiltonian is
\begin{equation*}
H_{2,j}^{\textup{red}}(q_{1},p_{1}) := p_{1}^{2} - q_{1}^{4} + q_{1}^{2}j.
\end{equation*}
The change in stability of the singularities of $H_{2,j}^{\textup{red}}$ is illustrated in Figure \ref{fig:type-H2}, and the image of the momentum map (and the bifurcation diagram) is drawn in Figure \ref{fig:bif-H2}. 

\begin{figure}[t]
\def\scale{0.75}
\begin{subfigure}{0.45\linewidth}
    \centering
    \includegraphics[width=0.9\linewidth]{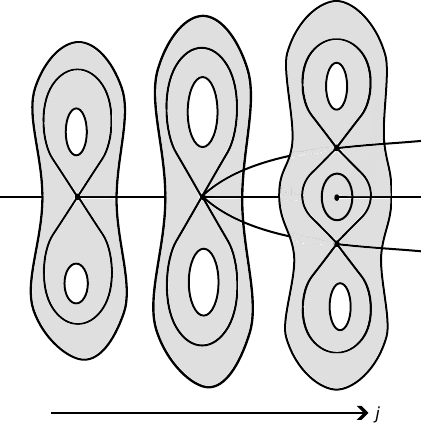}
    \caption{The types of singularities of the $H_{2,j}^{\textup{red}}$. Note that the upper and lower half are identified by the $\Z_{2}$-isotropy group.}
    \label{fig:type-H2}
\end{subfigure}
\begin{subfigure}{0.45\linewidth}
    \centering
    \includegraphics[width=0.9\linewidth]{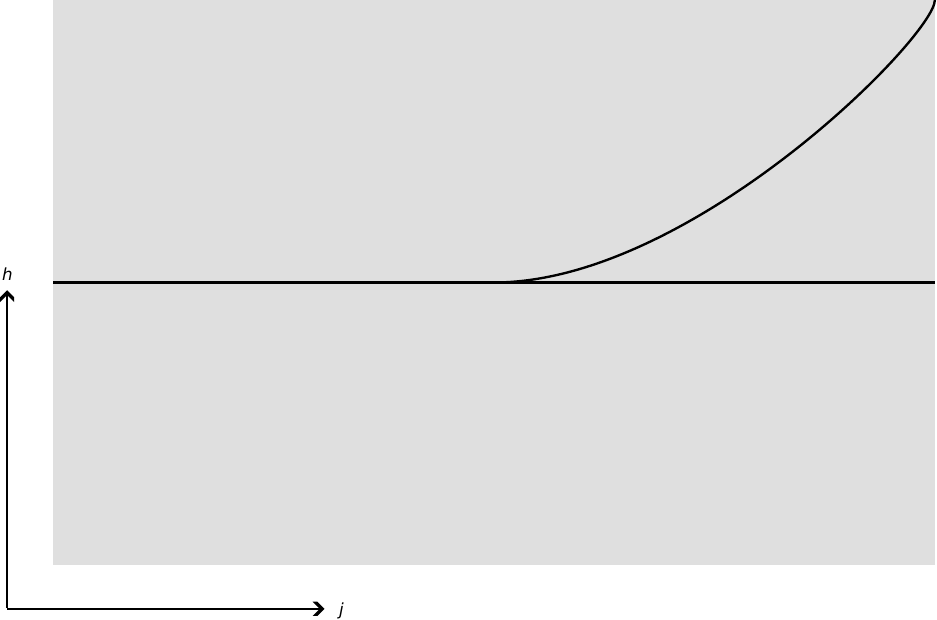}
    \caption{The black line and curve are the bifurcation diagram, while the grey area consist of regular values.}
    \label{fig:bif-H2}
\end{subfigure}
\caption{}
\label{fig:H2}
\end{figure}

Just as in the previous model, the branching point in Figure \ref{fig:H2} is a degenerate singularity. In Kalashnikov's list, this is the $2_{b}$-singularity. 

When an integrable system goes through a double flip bifurcation, the bifurcation diagram contains either two copies of the local model $H_{1}$, two copies of the local model $H_{2}$, or one copy of $H_{1}$ and one copy of $H_{2}$. We now give a normal form for Hamiltonian systems going through a double flip bifurcation. We also gives conditions for when the two Hamiltonian flip bifurcations are dual or not.

\begin{theorem} 
Assume that there are coordinate changes by smooth functions such that the $6$-jet of a function $f_{j,t}(q,p)$ depending on two parameters $j$ and $t$ can be put in the following form:
\begin{equation} \label{eq:H-intro}
H_{j,t}(q,p) = \frac{a}{2}p^{2} + \frac{b}{6}q^{6} + \frac{\nu_{1}(j,t)}{2}q^{2} + \frac{\nu_{2}(j,t)}{4}q^{4}
\end{equation}
where $a,b \in \R \setminus \{0\}$. Furthermore, we assume $\nu_{1}(j,t)$ satisfies
\begin{align} \label{list:saddle-node}
\begin{cases}
\nu_{1}(0,0) = 0, \\
\frac{\partial \nu_{1}}{\partial j}(0,0) = 0, \\
\frac{\partial^{2} \nu_{1}}{\partial j^{2}}(0,0) \neq 0, \\
\frac{\partial \nu_{1}}{\partial t}(0,0) \neq 0,
\end{cases}
\end{align}
and $\nu_{2}(0,0) = 0$. Then the system given by $f_{j,t}(q,p)$ goes through a double flip bifurcation as $t$ passes through $0$.

For $t = \tau$ sufficiently close to $0$, let $j_{0}^{\pm} = j_{0}^{\pm}(\tau)$ be such that $\nu_{1}(j_{0}^{\pm},\tau) = 0$ and $\nu_{2}(j_{0}^{\pm},\tau) \neq 0$. Then the Hamiltonian flip bifurcations occur for $(j,t) = (j_{0}^{\pm},\tau)$. If $a \nu_{2}(j_{0}^{\pm},\tau) < 0$, then the bifurcation is dual, and if $a \nu_{2}(j_{0}^{\pm},\tau) > 0$ it is not.
\end{theorem}

The theorem is restated as Theorem \ref{thm:normal-form}, and proven in Section \ref{sec:normal-form}.

Note that the Hamiltonian \eqref{eq:H-intro} with $\nu_{1}(j,t) = j$ and $\nu_{2}(j,t) = t$ is the universal unfolding of the $\Z_{2}$-symmetric $A_{5}^{\pm}$-singularity, see Han{\ss}mann \cite[Chapter 2]{Hanssmann2006}. The superscript $\pm$ is the sign of the product $ab$. To obtain Hamiltonian flip bifurcations one requires different conditions on $\nu_{1}(j,t)$ and $\nu_{2}(j,t)$. In fact, the list of conditions \eqref{list:saddle-node} are exactly the conditions for the differential equation $\frac{dj}{dt} = \nu_{1}(j,t)$ to go through a saddle-node bifurcation at the origin. This means that the system $\frac{dj}{dt} = \nu_{1}(j,t)$ is topologically equivalent to the system $\frac{dj}{dt} = j^{2} \pm t$. Thus, the equilibria for the system in normal form are given by $j^{2} \pm t = 0$. As $t$ passes through zero, $j$ changes from having zero to two solutions (or opposite), and these two solutions correspond to the two Hamiltonian flip bifurcations of the double flip bifurcation.

Let us now introduce another important element to understanding the double flip bifurcation:
\begin{equation*}
\disc_{b}(j,t) := \nu_{2}(j,t)^{2} - 4b\nu_{1}(j,t).
\end{equation*}
This is introduced in Section \ref{sec:bif-diag} where we discuss the critical points of the system. It turns out that also the differential equation $\frac{dj}{dt} = \disc_{b}(j,t)$ goes through a saddle-node bifurcation at the origin. However, one of the conditions is
\begin{equation} \label{eq:deg}
\frac{\partial^{2} \disc_{b}}{\partial j^{2}}(0,0) = 2 \left( \frac{\partial \nu_{2}}{\partial j}(0,0) \right)^{2} - 4b \frac{\partial^{2} \nu_{1}}{\partial j^{2}}(0,0).
\end{equation}
If \eqref{eq:deg} vanishes, then the saddle-node bifurcation is degenerate. We also say that the double flip bifurcation is degenerate in this case. If \eqref{eq:deg} does not vanish, then both bifurcations are non-degenerate. We now give a condition determining the direction of concavity of the two saddle-node bifurcations (for more details, see the paragraph above Theorem \ref{thm:concave}), which also yields the degeneracy of the double flip bifurcation. The direction of concavity is interesting as it yields information on how the bifurcation occurs (see Figures \ref{fig:bif-diag-normal-form-1}, \ref{fig:bif-diag-normal-form-2}, \ref{fig:bif-diag-normal-form-3}, \ref{fig:bif-diag-normal-form-4}, \ref{fig:osc-bif-diag}, \ref{fig:osc2-bif-diag}, and \ref{fig:W2-bif-diag}).

\begin{theorem}
The saddle-node bifurcations associated to $\nu_{1}(j,t)$ and $\disc_{b}(j,t)$ are concave the same direction if
\begin{equation*}
\frac{1}{2b} \frac{(\partial \nu_{2}/\partial j (0,0))^{2}}{\partial^{2} \nu_{1}/\partial j^{2} (0,0)} < 1.
\end{equation*}
If the inequality is reversed, they are concave in the opposite direction. If it is an equality, the double flip bifurcation is degenerate.
\end{theorem}

This theorem is restated as Theorem \ref{thm:concave}, and proven in Section \ref{sec:bif-diag}.

\subsection*{Overview}

In Section \ref{sec:HFlip}, we recall some facts about the Hamiltonian flip bifurcation. In particular we show that the Hamiltonian \eqref{eq:H-intro} satisfying \eqref{list:saddle-node} includes two (dual) Hamiltonian flip bifurcations for $t$ close enough to zero. In Section \ref{sec:normal-form}, we show that \eqref{eq:H-intro} in fact is a normal form for the double flip bifurcation. In Section \ref{sec:bif-diag}, we study the bifurcation diagram in more detail, and develop a way to tell whether the double flip bifurcation is degenerate or not. Finally, in Section \ref{sec:examples}, we study three examples.

\subsection*{Acknowledgments}

The second author was fully, and the third author partially supported by the FWO-EoS project \textit{Beyond symplectic geometry} with UA Antigoon number 45816. Moreover, the third author was also partially supported by the grant \textit{Francqui Research Professor 2023-2026} of the Francqui Foundation with UA Antigoon number 49741.
\section{The Hamiltonian flip bifurcation} \label{sec:HFlip}

In this section we define and describe (dual) Hamiltonian flip bifurcations in more detail. In particular, we show that the Hamiltonian \eqref{eq:H-intro} with $\nu_{1}$ satisfying the conditions \eqref{list:saddle-node} contains two Hamiltonian flip bifurcations, which may or may not be dual, as $t$ passes through some value $t_{0}$. We follow a similar procedure to that in \cite[pp. 33-34]{Hanssmann2006}.

Let us introduce a Hilbert basis for the $\Z_{2}$ action $(q,p) \mapsto (-q,-p)$: $u = \frac{1}{2}q^{2}$, $v = \frac{1}{2}p^{2}$, and $w = pq$. Note that these generators are related by the syzygy $4uv = w^{2}$. In terms of the coordinates $u$, $v$, and $w$, the Hamiltonian from Equation \eqref{eq:H-intro}
\begin{equation*}
H_{j,t}(q,p) = \frac{a}{2}p^{2} + \frac{b}{6}q^{6} + \frac{\nu_{1}(j,t)}{2}q^{2} + \frac{\nu_{2}(j,t)}{4}q^{4},
\end{equation*}
takes the following form:
\begin{equation*}
H_{j,t}(u,v,w) = av + \frac{4}{3}bu^{3} + \nu_{1}(j,t)u + \nu_{2}(j,t)u^{2}.
\end{equation*}
Furthermore, the Hilbert basis gives rise to a cone
\begin{equation*}
\mathcal{B} := \{ (u,v,w) \in \R^{3} : 4uv = w^{2}, \, u \geq 0, \, v \geq 0 \}.
\end{equation*}
Note that the level set $H_{j,t}^{-1}(0)$ contains the apex of the cone, which is a (conical) singular point, and hence a singular point of the Hamiltonian system. This point plays an important role for us in what follows, see in particular Theorem \ref{thm:HFlip} below.

Apart from the singularity at the apex of $\mathcal{B}$, all other singularities are characterised by the level sets $H_{j,t}^{-1}(h)$, for some $h \in \R$, intersecting the cone $\mathcal{B}$ non-transversally. Note that $\mathcal{B}$ is a surface of revolution. Furthermore, all level sets $H_{j,t}^{-1}(h)$ are independent of $w$. Hence, the tangent spaces to $H_{j,t}^{-1}(h)$ are planes parallel to the $w$-axis. This means all singularities require $w = 0$. Thus, we need only consider the cone $\mathcal{B}$ and level sets $H_{j,t}^{-1}(h)$ projected onto the plane $\{w=0\}$. We draw the cone together with certain level sets in Figure \ref{fig:HFlip}.

\begin{figure}[b]
\def\scale{0.35}
\centering
\begin{tabular}{ccc} 
 \subfloat{\includegraphics[scale=\scale]{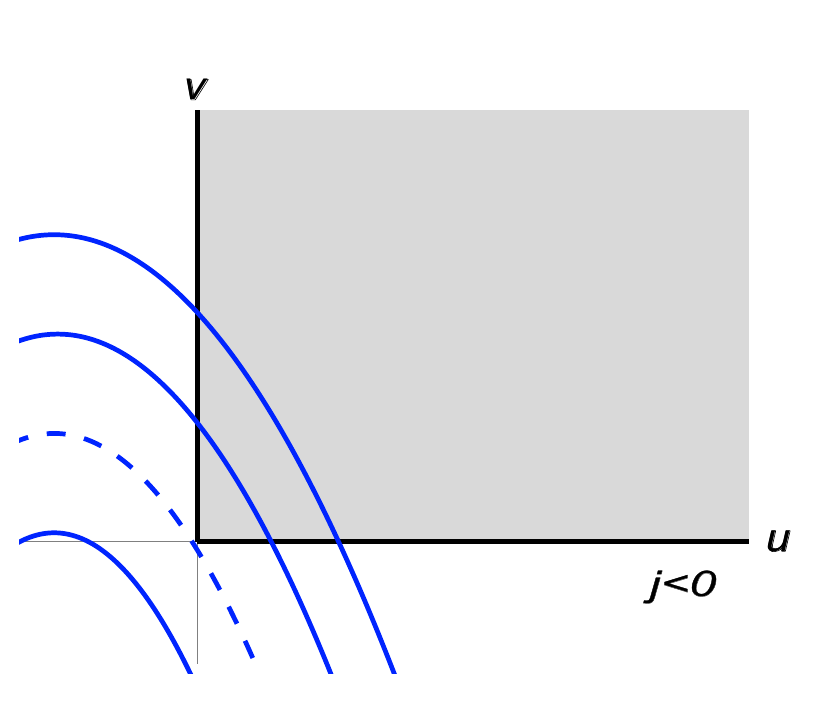}} & 
 \subfloat{\includegraphics[scale=\scale]{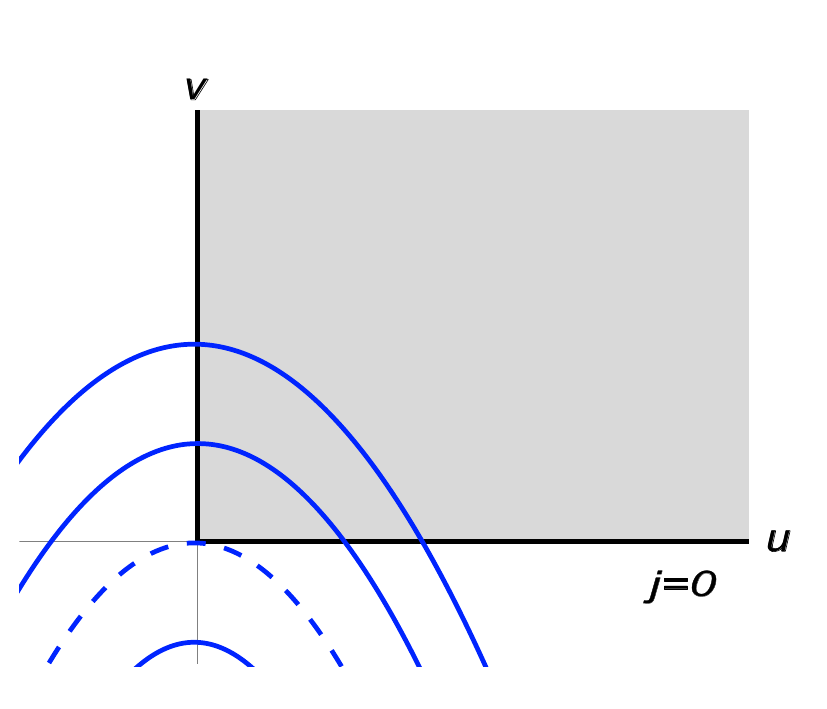}} & 
 \subfloat{\includegraphics[scale=\scale]{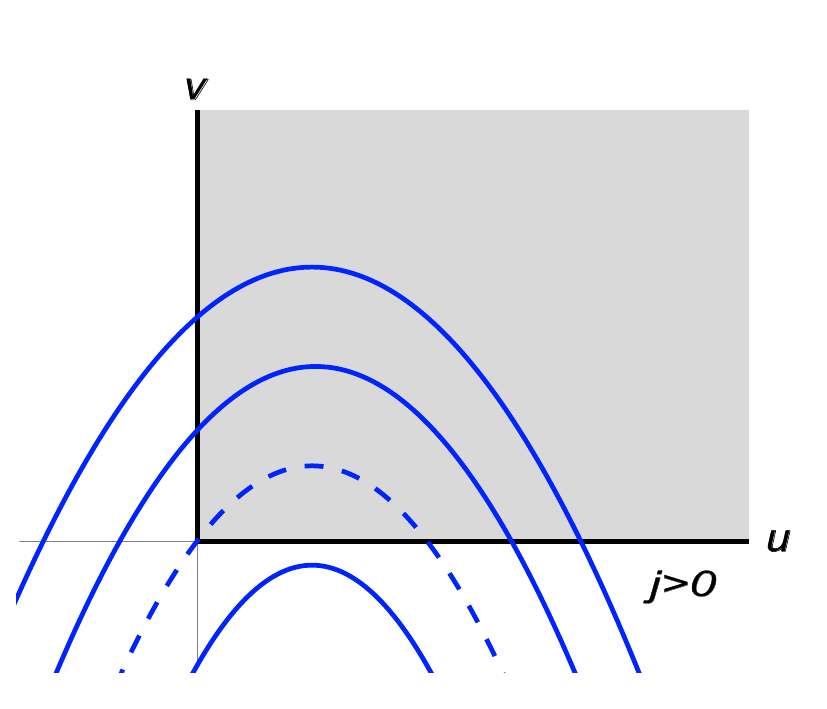}} \\
 \subfloat{\includegraphics[scale=\scale]{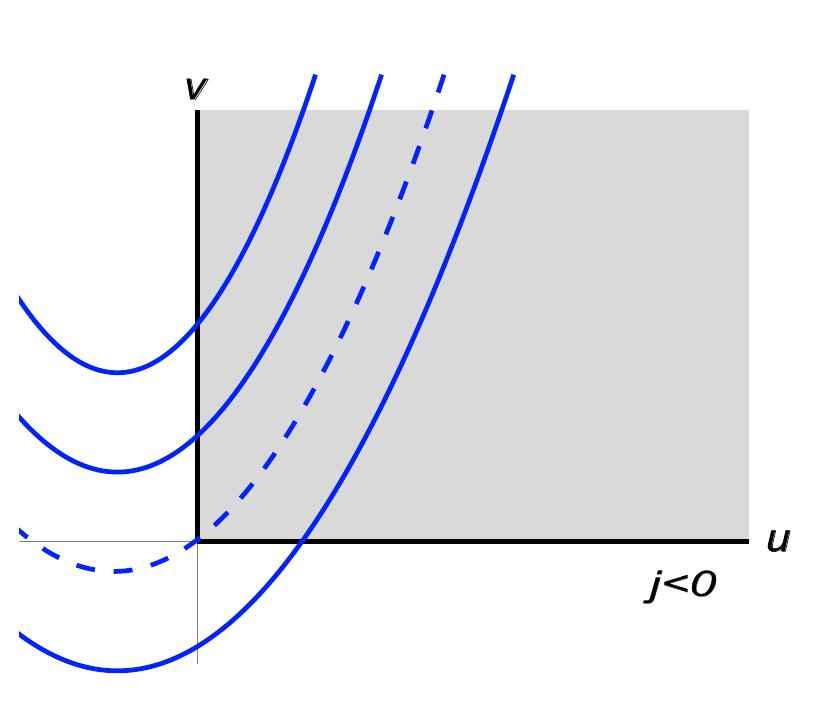}} & 
 \subfloat{\includegraphics[scale=\scale]{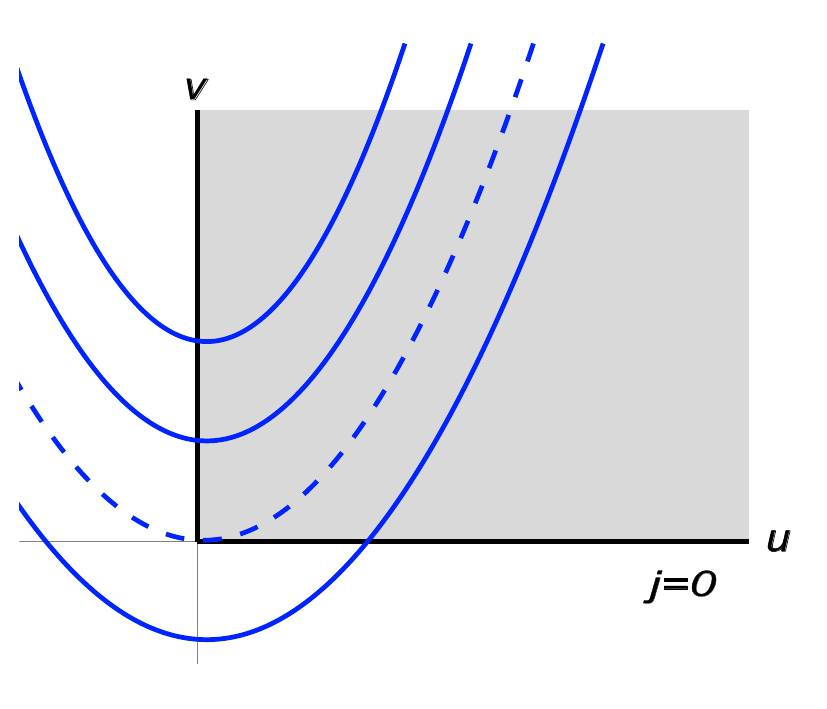}} & 
 \subfloat{\includegraphics[scale=\scale]{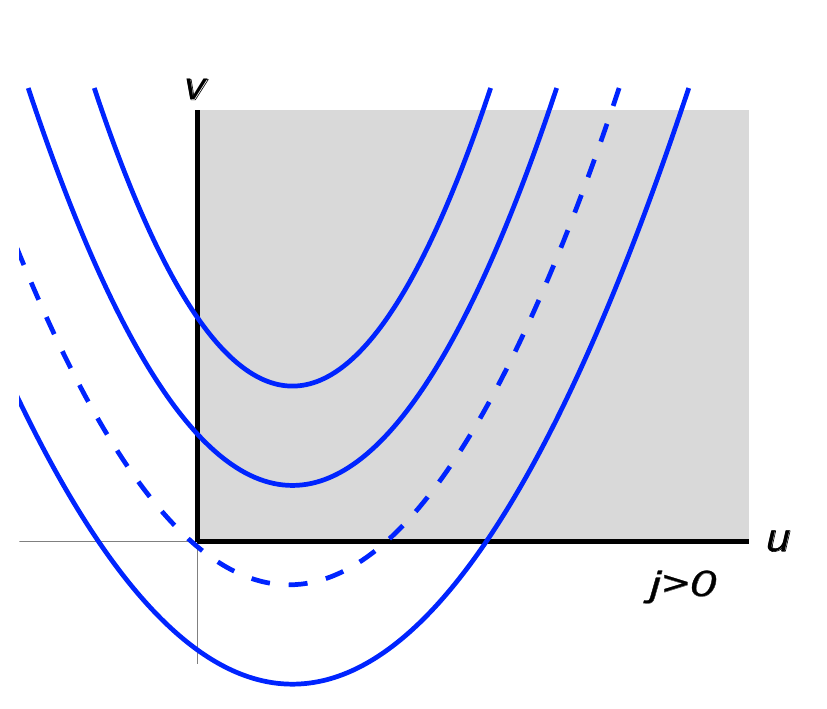}}
\end{tabular}
\caption{The grey area is the cone $\mathcal{B}$ projected to the plane $\{w=0\}$, and the blue curves are level sets $H_{j,t}^{-1}(h)$ for various levels $h$. The level set $H_{j,t}^{-1}(0)$ is drawn as a dashed curve. The top row shows a Hamiltonian flip bifurcation, as in the middle-figure $H_{j,t}^{-1}(0)$ only intersects $\mathcal{B}$ at its apex. The bottom row shows a dual Hamiltonian flip bifurcation, as in the middle-figure $H_{j,t}^{-1}(0)$ has further intersections with $\mathcal{B}$.}
\label{fig:HFlip}
\end{figure}

A singular point $x$ is called \emph{stable} (in the sense of Lyapunov) if all points in a sufficiently small neighbourhood of $x$ remain close to $x$ under the flow of the system for all time. If a singular point is not stable, then it is called \emph{unstable}. If all eigenvalues of the linearised system are non-zero and on the imaginary axis, then $x$ is a \emph{centre}, and $x$ is stable. If the linearised system at $x$ has no purely imaginary eigenvalues (where $0$ is counted as purely imaginary), then $x$ is a \emph{saddle}, and $x$ is unstable.

\begin{definition} \label{def:HFlip}
Consider a Hamiltonian system depending on a bifurcational parameter $j$, and let $0 \in \mathcal{B}$ be a singularity of the system. 
\begin{itemize}
    \item Assume $0$ is stable. If $0$ becomes unstable as $j$ goes through a value $j_{0}$, and at the same time produces an additional singularity, which is a centre, then $0$ goes through a \emph{Hamiltonian flip bifurcation} at $j = j_{0}$.
    \item Assume $0$ is unstable. If $0$ becomes stable as $j$ goes through a value $j_{0}$, and at the same time produces an additional singularity, which is a saddle, then $0$ goes through a \emph{dual Hamiltonian flip bifurcation} at $j = j_{0}$.
\end{itemize}
\end{definition}

The following theorem gives a method to determine if a (dual) Hamiltonian flip bifurcation takes place in a Hamiltonian system. 

\begin{theorem}{\cite[Theorem 2.17]{Hanssmann2006}} \label{thm:HFlip}
Let the $2$-dimensional phase space have a conical singular point, and let a neighbourhood of the singular point be called a local cone. Consider a $1$-parameter family of Hamiltonian systems given by the Hamiltonian $H_{j}$, and assume that the level sets $H_{j}^{-1}(h)$ intersects the local cone transversally. The level set of the singular point, $H_{j}^{-1}(0)$, intersects the local cone for $j < 0$, and for each $j > 0$ there is a neighbourhood of the origin in which the local cone has no further intersection with $H_{j}^{-1}(0)$. Then 
\begin{itemize}
    \item a Hamiltonian flip bifurcation occurs if $H_{0}^{-1}(0)$ touches the local cone from the outside (see the top row of Figure \ref{fig:HFlip}), and 
    \item a dual Hamiltonian flip bifurcation occurs if there are additional intersection points of $H_{0}^{-1}(0)$ with the local cone (see the bottom row of Figure \ref{fig:HFlip}).
\end{itemize}
\end{theorem}

In our system $H_{j,t}$, the parameter governing the Hamiltonian flip bifurcations is the parameter $j$. When we consider Hamiltonian flip bifurcations in this system, we consider the system with some fixed $t$.

\begin{lemma} \label{lem:H-flip}
Let $H_{j,t}$ be as in \eqref{eq:H-intro} with $\nu_{1}(j,t)$ satisfying \eqref{list:saddle-node}. Fix $\tau \in \R$ sufficiently close to $0$, and let $j_{0}^{\pm} = j_{0}^{\pm}(\tau)$ be real, and such that $\nu_{1}(j_{0}^{\pm},\tau) = 0$. Furthermore, assume that $\nu_{2}(j_{0}^{\pm},\tau) \neq 0$. Then Hamiltonian flip bifurcations occur for $(j,t) = (j_{0}^{\pm},\tau)$. 

Recall the coefficient $a$ from $H_{j,t}$. If $a \nu_{2}(j_{0}^{\pm},\tau) < 0$, then the bifurcation is dual, and if $a \nu_{2}(j_{0}^{\pm},\tau) > 0$ it is not.
\end{lemma}

\begin{proof}
To detect a Hamiltonian flip bifurcation it suffices to study the level sets $H_{j,t}^{-1}(h)$ near the apex of the cone. Indeed, by Theorem \ref{thm:HFlip}, we need to show that the level sets $H_{j,t}^{-1}(h)$ intersects the local cone transversally, apart from $H_{j,t}^{-1}(0)$, which touches the apex of the cone. Furthermore, if the slope of $H_{j,t}^{-1}(0)$ is positive there are more intersection points with the local cone, while if it is negative there is a neighbourhood of the apex with no more intersections.

The level set $H_{j,t}^{-1}(h)$ can be written as a function of $u$ by solving $H_{j,t}(u,v,w) = h$ for $v$. We denote this solution by $v_{h}(u,j,t)$:
\begin{equation*}
v_{h}(u,j,t) = \frac{1}{a} \left( h - \frac{4b}{3}u^{3} - \nu_{1}(j,t)u - \nu_{2}(j,t)u^{2} \right).
\end{equation*}
Its derivative with respect to $u$ is
\begin{equation*}
\frac{\partial v_{h}}{\partial u}(u,j,t) = -\frac{1}{a} (4bu^{2} + \nu_{1}(j,t) + 2\nu_{2}(j,t)u).
\end{equation*}
The apex occurs at $u = 0$, hence we calculate
\begin{align*}
\frac{\partial v_{h}}{\partial u}(0,j,t) = - \frac{1}{a}\nu_{1}(j,t),
\end{align*}
which vanishes at $(j,t) = (j_{0}^{\pm},\tau)$. By assumption $\frac{\partial \nu_{1}}{\partial j}(0,0) = 0$, $\frac{\partial^{2} \nu_{1}}{\partial j^{2}}(0,0) \neq 0$, and $\frac{\partial \nu_{1}}{\partial t}(0,0) \neq 0$. Thus, in a neighbourhood of the origin, the solution to $\nu_{1}(j,t) = 0$ has the shape of a parabola. As we pass from $j < j_{0}^{+}$ to $j > j_{0}^{+}$ (respectively from $j < j_{0}^{-}$ to $j > j_{0}^{-}$) we pass through $\nu_{1}(j_{0}^{+},\tau) = 0$ (respectively $\nu_{1}(j_{0}^{-},\tau) = 0$), i.e.\ $\nu_{1}(j,\tau)$ changes sign, and hence so does $\frac{\partial v_{h}}{\partial u}(0,j,\tau)$. 

To check if the level sets intersects the local cone transversally, and whether the bifurcation is dual or not, we consider the second derivative of $v_{0}(u,j,t)$ evaluated at $(u,j,t) = (0,j_{0}^{\pm},\tau)$. We find
\begin{equation*}
\frac{\partial^{2} v_{h}}{\partial u^{2}}(u,j,t) = -\frac{2}{a} (4bu + \nu_{2}(j,t))
\implies
\frac{\partial^{2} v_{0}}{\partial u^{2}}(0,j_{0}^{\pm},\tau) = - \frac{2}{a} \nu_{2}(j_{0}^{\pm},\tau).
\end{equation*}
Since we assume $\nu_{2}(j_{0}^{\pm},\tau) \neq 0$ the second derivative does not vanish, and the level sets intersects the local cone transversally. If $\frac{\partial^{2} v_{0}}{\partial u^{2}}(0,j_{0}^{\pm},\tau)$ is positive, then the level set $H_{j,t}^{-1}(0)$ is a concave-up parabola, which happens if $a \nu_{2}(j_{0}^{\pm},\tau) < 0$. In this case there are additional intersection points of $H_{j_{0}^{\pm},\tau}^{-1}(0)$ with the local cone. If $\frac{\partial^{2} v_{0}}{\partial u^{2}}(0,j_{0}^{\pm},\tau)$ is negative, then the level set is a concave-down parabola, which happens if $a \nu_{2}(j_{0}^{\pm},\tau) > 0$, and there are no further intersection with the local cone.
\end{proof}
\section{Normal form} \label{sec:normal-form}

In this section we find a normal form for Hamiltonian systems going through a bifurcation which we name the double flip bifurcation. This is a bifurcation which produces two Hamiltonian flip bifurcations in the bifurcation diagram. More precisely, the system depends on two parameters $t$ and $j$, where $t$ is the parameter governing the double flip bifurcation, and $j$ is the parameter governing the Hamiltonian flip bifurcations. As $t$ passes through the point of bifurcation, then two Hamiltonian flip bifurcations with respect to $j$ are introduced in the system.

We find a normal form for Hamiltonians going through double flip bifurcations by showing that all functions with a $6$-jet of the same form as \eqref{eq:H-intro}, with $\nu_{1}$ satisfying the conditions \eqref{list:saddle-node}, are equivalent in the sense defined below. This we achieve by following the procedure described by Mather \cite{Mather1968} and Gibson \cite[Section IV]{Gibson1979}.

Let $M,N$ be manifolds, and let $x \in M$ be a point in $M$. Two mappings $f_{i} : M \to N$, with $i \in \{1,2\}$, define the same \emph{germ at $x$} if there exists a neighbourhood $U \subset M$ of $x$ such that $f_{1}|_{U} = f_{2}|_{U}$. A germ $f : M \to N$ is \emph{invertible} if there exists a germ $g : N \to M$ such that $f \circ g = 1_{N}$ and $g \circ f = 1_{M}$, where $1_{M}$ (respectively $1_{N}$) is the identity in $M$ (respectively $N$).

\begin{definition}
Let $M_{1}$, $M_{2}$ and $N$ be manifolds. We call two germs of functions $f_{i} : M_{i} \to N$, with $i \in \{1,2\}$, \emph{equivalent} if there exists an invertible germ $g : M_{2} \to M_{1}$ such that $f_{1} \circ g = f_{2}$.
\end{definition}

We study germs $f : M \to \R$ at $x \in M$, and therefore only the local behaviour of $f$ is of concern. Hence, we may simply consider a chart at $x$, allowing us to assume $M = \R^{n}$, with $\dim M = n$. Thus, we will instead study germs $f : \R^{n} \to \R$ at $0 \in \R^{n}$. In particular, we are interested in the case $n = 2$. Let us denote by $\mathcal{E}$ the algebra of all germs $f : \R^{2} \to \R$ at $0$. To classify the germs in $\mathcal{E}$ one often considers their Taylor expansion. The Taylor expansion of a germ $f \in \mathcal{E}$ truncated at $k$-th order is called the \emph{$k$-jet} of $f$. 

\begin{definition}
A germ $f \in \mathcal{E}$ is said to be \emph{$k$-determined} when any germ $g \in \mathcal{E}$ with same $k$-jet as $f$ is equivalent to $f$.
\end{definition}

Below we recall a theorem giving a condition for $k$-determinacy of germs. The theorem makes use of two ideals in the algebra $\mathcal{E}$. Namely, we are interested in the Jacobian ideal $\mathcal{J}_{f}$, generated by the partial derivatives of $f$, and the maximal ideal $\mathcal{M}$ generated by the monomials in $\mathcal{E}$. In coordinates $\{x_{1},x_{2},\dots\}$ on $\R^{n}$, the ideals are written in terms of their generators as
\begin{align*}
\mathcal{J}_{f} = \left\langle \frac{\partial f}{\partial x_{1}}, \frac{\partial f}{\partial x_{2}}, \dots \right\rangle 
\quad
\text{and}
\quad
\mathcal{M} = \langle x_{1}, x_{2}, \dots \rangle.
\end{align*}
We also introduce the ideal $\mathcal{M}^{k}$, which is generated by all monomials of degree $k$.

\begin{theorem}[{\cite[Theorem 3.1, p.\ 117]{Gibson1979}}]
Let $f \in \mathcal{E}$ be such that $\mathcal{M}^{k} \subset \mathcal{M}\mathcal{J}_{f}$. Then $f$ is $k$-determined.
\end{theorem}

Using this theorem we prove the following lemma:

\begin{lemma} \label{lem:determinacy}
Let $j$ and $t$ be two parameters. The Hamiltonian 
\begin{equation*}
H_{j,t}(q,p) = \frac{a}{2}p^{2} + \frac{b}{6}q^{6} + \frac{\nu_{1}(j,t)}{2}q^{2} + \frac{\nu_{2}(j,t)}{4}q^{4},
\end{equation*}
for $a,b \in \R \setminus \{0\}$ and $\nu_{1},\nu_{2} : \R^{2} \to \R$ smooth functions such that $\nu_{1}(0,0) = 0 = \nu_{2}(0,0)$, is $6$-determined.
\end{lemma}

\begin{remark}
In Lemma \ref{lem:determinacy} the additional conditions on $\nu_{1}(j,t)$ from \eqref{list:saddle-node} are not necessary.
\end{remark}

\begin{proof}[Proof of Lemma \ref{lem:determinacy}]
Here the Jacobian ideal is
\begin{equation*}
\mathcal{J}_{H_{j,t}}
= \langle ap,\underbrace{bq^{5}+\nu_{1}(j,t)q+\nu_{2}(j,t)q^{3}}_{=: f_{j,t}(q)} \rangle.
\end{equation*}
Thus, the product ideal is 
\begin{equation*}
\mathcal{M}\mathcal{J}_{H_{j,t}} = \langle p,q \rangle \langle ap,f_{j,t}(q) \rangle.
\end{equation*}
The first factor of $\mathcal{J}_{H_{j,t}}$, namely $ap$, combined with $\mathcal{M}$ yields all generators of $\mathcal{M}^{6}$ but $q^{6}$. Hence, we only need to show that the ideal $\langle f_{j,t}(q) \rangle \langle p,q \rangle$ contains $q^{6}$. Let us denote by $g(q)$ the generator $g(q) = qf_{j,t}(q) \in \mathcal{M}\mathcal{J}_{H_{j,t}}$, and consider the germ
\begin{equation*}
F_{k}(q) = \frac{q^{k}}{g(q)} = \frac{q^{k}}{bq^{6} + \nu_{1}(j,t)q^{2} + \nu_{2}(j,t)q^{4}}.
\end{equation*}
This is a smooth germ in $\mathcal{E}$ for all of $\R^{2}$ for $k \geq 6$ since $b \neq 0$. For $k < 6$ it is not smooth near $q = 0$ since $\nu_{1}(0,0) = 0 = \nu_{2}(0,0)$. Furthermore,
\begin{equation*}
q^{6} = F_{6}(q) g(q),
\end{equation*}
so also $q^{6} \in \mathcal{M}J_{H_{j,t}}$. Hence, $\mathcal{M}^{6} \subset \mathcal{M}J_{H_{j,t}}$, and the lemma has been proven.
\end{proof}

The Lemmas \ref{lem:H-flip} and \ref{lem:determinacy} imply the following theorem:

\begin{theorem} \label{thm:normal-form}
Assume that there are coordinate changes by smooth functions such that the $6$-jet of a function $f_{j,t}(q,p)$ depending on two parameters $j$ and $t$ can be put in the following form:
\begin{equation*}
H_{j,t}(q,p) = \frac{a}{2}p^{2} + \frac{b}{6}q^{6} + \frac{\nu_{1}(j,t)}{2}q^{2} + \frac{\nu_{2}(j,t)}{4}q^{4}
\end{equation*}
where $a,b \in \R \setminus \{0\}$, and $\nu_{1}(0,0) = 0 = \nu_{2}(0,0)$. Furthermore, we assume $\nu_{1}(j,t)$ satisfies
\begin{align*}
\begin{cases}
\frac{\partial \nu_{1}}{\partial j}(0,0) = 0, \\
\frac{\partial^{2} \nu_{1}}{\partial j^{2}}(0,0) \neq 0, \\
\frac{\partial \nu_{1}}{\partial t}(0,0) \neq 0.
\end{cases}
\end{align*}
Then the system given by $f_{j,t}(q,p)$ goes through a double flip bifurcation as $t$ passes through $0$.

For $t = \tau$ sufficiently close to $0$, let $j_{0}^{\pm} = j_{0}^{\pm}(\tau)$ be such that $\nu_{1}(j_{0}^{\pm},\tau) = 0$ and $\nu_{2}(j_{0}^{\pm},\tau) \neq 0$. Then the Hamiltonian flip bifurcations occur for $(j,t) = (j_{0}^{\pm},\tau)$. If $a \nu_{2}(j_{0}^{\pm},\tau) < 0$, then the bifurcation is dual, and if $a \nu_{2}(j_{0}^{\pm},\tau) > 0$ it is not.
\end{theorem}
\section{The bifurcation diagram} \label{sec:bif-diag}

In this section we study the bifurcation diagram for the Hamiltonian \eqref{eq:H-intro}. To this end, let us consider the derivative of $H_{j,t}(q,p)$:
\begin{equation*}
dH_{j,t}(q,p) = \left( \frac{\partial H_{j,t}}{\partial q}, \frac{\partial H_{j,t}}{\partial p} \right) = ( bq^{5} + \nu_{1}(j,t)q + \nu_{2}(j,t)q^{3}, ap ).
\end{equation*}
Thus, any critical point must satisfy $p = 0$. Furthermore, one solution for the first component to vanish is $q = q_{1} := 0$. The other critical points are given by the solutions to the following equation:
\begin{equation} \label{eq:dH/q}
bq^{4} + \nu_{1}(j,t) + \nu_{2}(j,t)q^{2} = 0.
\end{equation}
Denote the discriminant by
\begin{equation*}
\disc_{b}(j,t) := \nu_{2}(j,t)^{2} - 4b\nu_{1}(j,t).
\end{equation*}
Then the solutions to \eqref{eq:dH/q} are given by
\begin{align*}
&q = q_{2}^{-}(j,t) := \pm \frac{1}{\sqrt{2b}} \sqrt{-\nu_{2}(j,t) - \sqrt{\disc_{b}(j,t)}}, \\
&q = q_{2}^{+}(j,t) := \pm \frac{1}{\sqrt{2b}} \sqrt{-\nu_{2}(j,t) + \sqrt{\disc_{b}(j,t)}}.
\end{align*}
Note that the sign of $q_{2}^{-}(j,t)$ and $q_{2}^{+}(j,t)$ do not play a big role for us, as all the $q$-terms of the Hamiltonian are of even power. We have found that the critical points are given by $(q,p) = (0,0)$, and $(q,p) = (q_{2}^{\pm}(j,t),0)$. 

Next, let us consider the type of the singularities, and how they depend on the parameters $j$ and $t$. This we do by computing the eigenvalues of the linearised system. Let $\Omega$ be the matrix of the canonical symplectic form. Then the linearised Hamiltonian vector field is
\begin{align*}
\Omega d^{2}H_{j,t}(q,p) &=
\begin{pmatrix}
0 & 1 \\
-1 & 0
\end{pmatrix}
\begin{pmatrix}
5bq^{4} + \nu_{1}(j,t) + 3\nu_{2}(j,t)q^{2} & 0 \\
0 & a
\end{pmatrix} \\&=
\begin{pmatrix}
0 & -a \\
5bq^{4} + \nu_{1}(j,t) + 3\nu_{2}(j,t)q^{2} & 0
\end{pmatrix}.
\end{align*}
We denote the eigenvalues for the critical point $(q,p) = (q_{1}(j,t),0) = (0,0)$ by $\lambda_{1}(j,t)$, and for the critical points $(q,p) = (q_{2}^{\pm}(j,t),0)$ by $\lambda_{2}^{\pm}(j,t)$. The eigenvalues are
\begin{align*}
&\lambda_{1}(j,t) = \pm i \sqrt{a\nu_{1}(j,t)}, \\
&\lambda_{2}^{-}(j,t) = \pm \sqrt{ \frac{a}{b} \left( 4b\nu_{1}(j,t) - \nu_{2}(j,t)(\sqrt{\disc_{b}(j,t)} + \nu_{2}(j,t)) \right) }, \\
&\lambda_{2}^{+}(j,t) = \pm \sqrt{ \frac{a}{b} \left( 4b\nu_{1}(j,t) + \nu_{2}(j,t)(\sqrt{\disc_{b}(j,t)} - \nu_{2}(j,t)) \right) }.
\end{align*}
Note that $\lambda_{1}(j,t)$ vanishes when $\nu_{1}(j,t)$ vanishes. Furthermore, if $\disc_{b}(j,t)$ vanishes, i.e.\ $4b\nu_{1}(j,t) = \nu_{2}(j,t)^{2}$, then so does $\lambda_{2}^{\pm}(j,t)$ vanish. Thus, the solutions to $\nu_{1}(j,t) = 0$ and $\disc_{b}(j,t) = 0$ yield degenerate singularities, as the eigenvalues vanish. Notice also that when $\disc_{b}(j,t) = 0$, then
\begin{equation*}
q_{2}^{-}(j,t)|_{\disc_{b}(j,t)=0} = q_{2}^{+}(j,t)|_{\disc_{b}(j,t)=0} = \pm \frac{1}{\sqrt{2b}}\sqrt{-\nu_{2}(j,t)}.
\end{equation*}
Since $q \in \R$, along $\disc_{b}(j,t) = 0$ we require $\nu_{2}(j,t) < 0$.

Now, let us study when the eigenvalues $\lambda_{1}(j,t)$ and $\lambda_{2}^{\pm}(j,t)$ are real (i.e.\ the singularity is a saddle) and imaginary (i.e.\ the singularity is a centre). It is immediate that $\lambda_{1}(j,t)$ is real if $a\nu_{1}(j,t) < 0$ and imaginary if $a\nu_{1}(j,t) > 0$. Next, note that, since $q_{2}^{\pm}(j,t) \in \R$, then $\disc_{b}(j,t) > 0$ and $\nu_{2}(j,t) < \pm \sqrt{\disc_{b}(j,t)}$ for all $j,t$. Thus, for all $j,t$ there exists some $\epsilon > 0$ such that $\nu_{2}(j,t) = \pm \sqrt{\disc_{b}(j,t)} - \epsilon$. Then the eigenvalues, up to sign, become
\begin{align*}
\lambda_{2}^{\pm}(j,t) &
= \sqrt{\frac{a}{b} \left( \nu_{2}(j,t)^{2} - \disc_{b}(j,t) \pm \nu_{2}(j,t) (\sqrt{\disc_{b}(j,t)} \mp \nu_{2}(j,t)) \right)} \\&
= \sqrt{\mp \frac{a}{b} \epsilon \sqrt{\disc_{b}(j,t)}}.
\end{align*}
Since $\epsilon \sqrt{\disc_{b}(j,t)} > 0$, we get that $\lambda_{2}^{+}(j,t)$ is real if $ab < 0$ and imaginary if $ab > 0$, and $\lambda_{2}^{-}(j,t)$ is real if $ab > 0$ and imaginary if $ab < 0$. Thus we have proven the following theorem:

\begin{theorem} \label{thm:sing-type}
Consider a $2$-parameter Hamiltonian with $6$-jet
\begin{equation*}
H_{j,t}(q,p) = \frac{a}{2}p^{2} + \frac{b}{6}q^{6} + \frac{\nu_{1}(j,t)}{2}q^{2} + \frac{\nu_{2}(j,t)}{4}q^{4},
\end{equation*}
where $a,b \in \R \setminus \{0\}$ and $\nu_{1},\nu_{2} : \R^{2} \to \R$. Then the critical points are $(q_{1}(j,t),0)$, and $(q_{2}^{\pm}(j,t),0)$, where
\begin{align*}
&q_{1}(j,t) = 0, \\
&q_{2}^{-}(j,t) = \pm \frac{1}{\sqrt{2b}} \sqrt{-\nu_{2}(j,t) - \sqrt{\disc_{b}(j,t)}}, \\
&q_{2}^{+}(j,t) = \pm \frac{1}{\sqrt{2b}} \sqrt{-\nu_{2}(j,t) + \sqrt{\disc_{b}(j,t)}}.
\end{align*}
The type of the critical points are as follows:
\begin{itemize}
    \item $(q_{1}(j,t),0)$ is a centre if $a\nu_{1}(j,t) > 0$ and a saddle if $a\nu_{1}(j,t) < 0$.
    \item For all $j,t$ the critical point $(q_{2}^{-}(j,t),0)$ is a saddle if $ab > 0$ and a centre if $ab < 0$.
    \item For all $j,t$ the critical point $(q_{2}^{+}(j,t),0)$ is a centre if $ab > 0$ and a saddle if $ab < 0$.
\end{itemize}
For $\nu_{1}(j,t) = 0$ the singularity $(q_{1}(j,t),0)$ is degenerate. For $\disc_{b}(j,t) = 0$ the singularity $(q_{2}^{-}(j,t),0) = (q_{2}^{+}(j,t),0)$ is degenerate.
\end{theorem}

In Figures \ref{fig:bif-diag-normal-form-1}, \ref{fig:bif-diag-normal-form-2}, \ref{fig:bif-diag-normal-form-3}, and \ref{fig:bif-diag-normal-form-4}, we draw bifurcation diagrams for certain choices of $\nu_{1}(j,t)$ and $\nu_{2}(j,t)$. In all figures the coefficients of $H_{j,t}(q,p)$ are $a = b = 1$. Note that in the figures (a) the location of the degenerate singularities is drawn. The solid curve corresponds to the solutions to $\nu_{1}(j,t) = 0$, and the dashed curve to the solutions to $\disc_{b}(j,t) = 0$. In Figure \ref{fig:bif-diag-normal-form-2} the dashed line corresponds to the negative part of the $j$-axis. Below we give a condition to determine when the curves corresponding to $\nu_{1}(j,t)$ and $\disc_{b}(j,t)$ are concave the same or the opposite direction. 

In the figures (b), (c), and (d), the bifurcation diagrams are drawn in red, to make the distinction to figures (a) clear. The solid lines correspond to the critical point $(q,p) = (0,0)$, the dashed curves to $(q,p) = (q_{2}^{-}(j,t),0)$, and the dashed-dotted curves to $(q,p) = (q_{2}^{+}(j,t),0)$. The type of the singularities can readily be found by Theorem \ref{thm:sing-type}. Indeed, in all figures $a = b = 1$, so $ab > 0$. Thus the dashed curves are saddles, the dashed-dotted are centres, while the solid line depends on $j$ and $t$. In all figures the whole solid line are centres for $t < 0$, while at $t = 0$ there are centres except for a degenerate singularity at the origin. In Figures \ref{fig:bif-diag-normal-form-1}, \ref{fig:bif-diag-normal-form-2}, and \ref{fig:bif-diag-normal-form-3}, the segment between the dashed and dashed-dotted curves are saddles, while the remainder apart from the degenerate singularities are centres. In Figure \ref{fig:bif-diag-normal-form-4} it is opposite; the middle segment are centres, and the sides saddles.


\begin{figure}[p]
\centering
\def\scale{0.35}

\begin{tabular}{cccc} 
 \subfloat[Degenerate singularities.]{\includegraphics[scale=\scale]{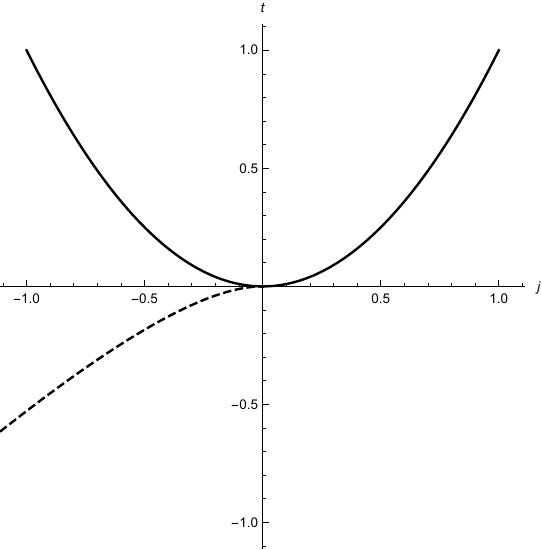}}
 & \subfloat[$t = -0.1$]{\includegraphics[scale=\scale]{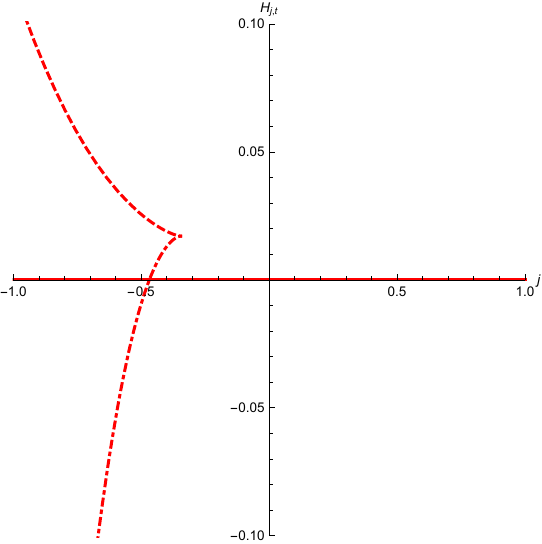}} 
 & \subfloat[$t = 0$]{\includegraphics[scale=\scale]{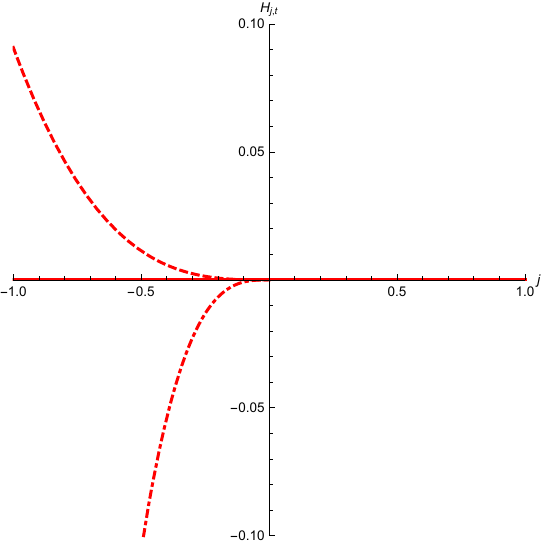}} 
 & \subfloat[$t = 0.1$]{\includegraphics[scale=\scale]{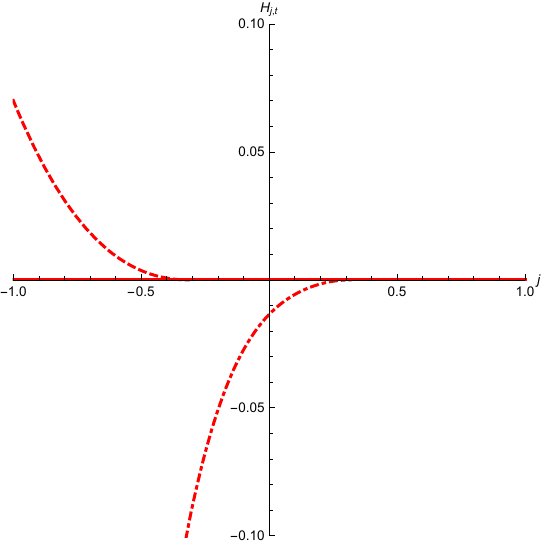}}
\end{tabular}
\caption{Here $a = b = 1$, and $\nu_{1}(j,t) = j^{2}-t$, $\nu_{2}(j,t) = 3j-t$.}
\label{fig:bif-diag-normal-form-1}

\begin{tabular}{cccc} 
 \subfloat[Degenerate singularities.]{\includegraphics[scale=\scale]{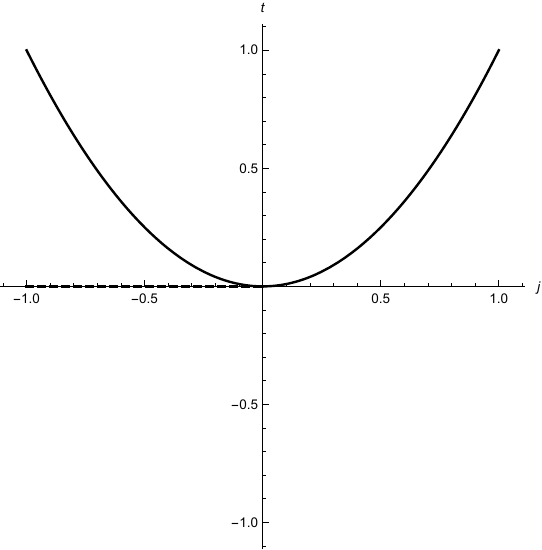}}
 & \subfloat[$t = -0.1$]{\includegraphics[scale=\scale]{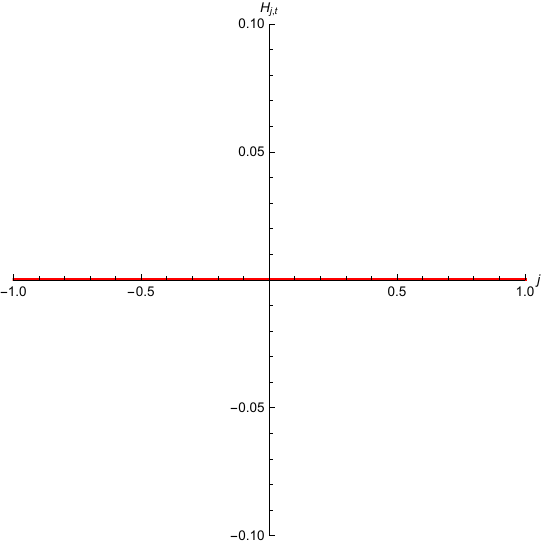}} 
 & \subfloat[$t = 0$]{\includegraphics[scale=\scale]{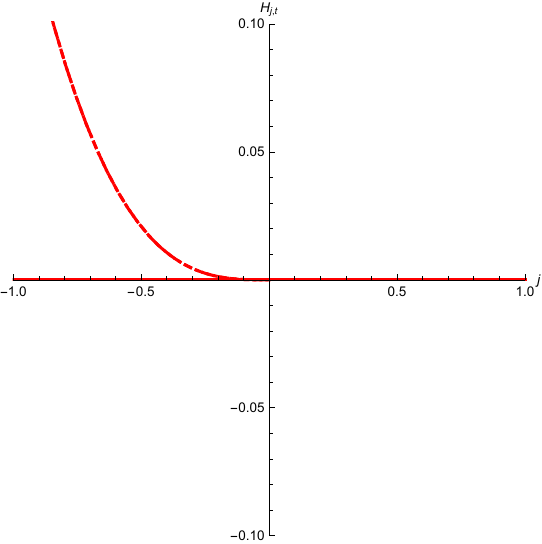}} 
 & \subfloat[$t = 0.1$]{\includegraphics[scale=\scale]{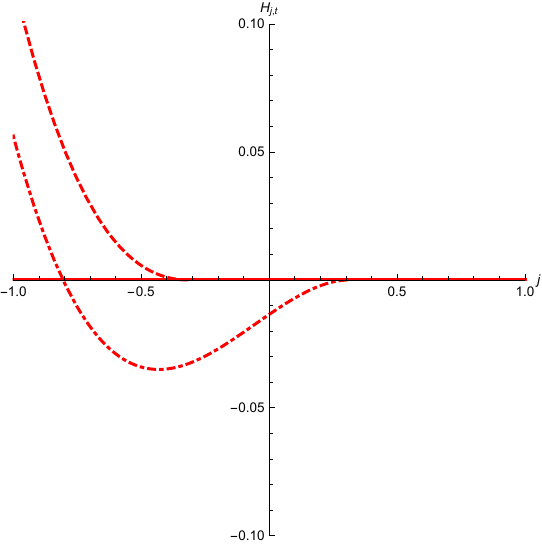}}
\end{tabular}
\caption{Here $a = b = 1$, and $\nu_{1}(j,t) = j^{2}-t$, $\nu_{2}(j,t) = 2j-t$.}
\label{fig:bif-diag-normal-form-2}

\begin{tabular}{cccc} 
 \subfloat[Degenerate singularities.]{\includegraphics[scale=\scale]{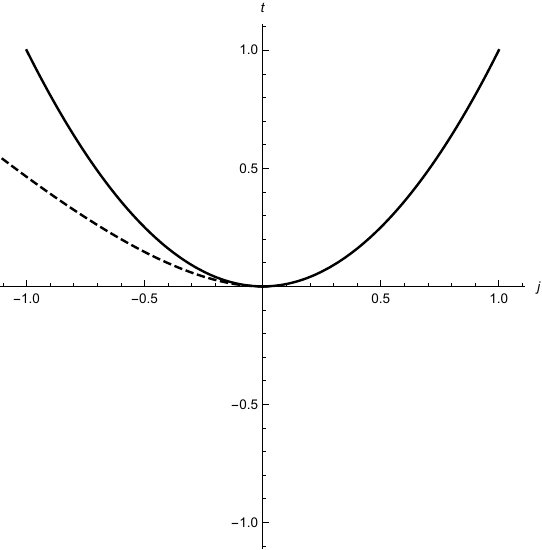}}
 & \subfloat[$t = -0.3$]{\includegraphics[scale=\scale]{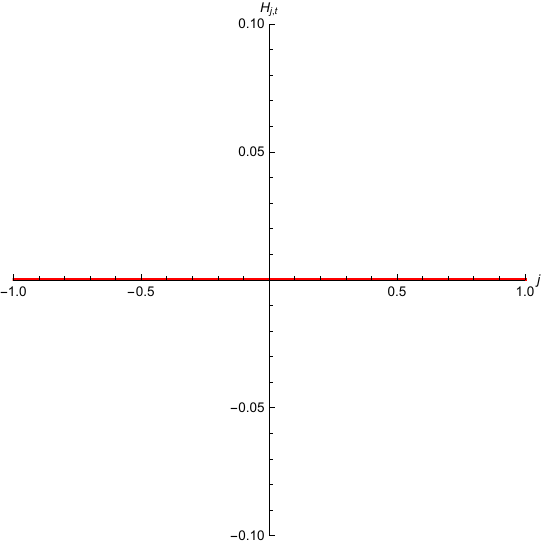}} 
 & \subfloat[$t = 0$]{\includegraphics[scale=\scale]{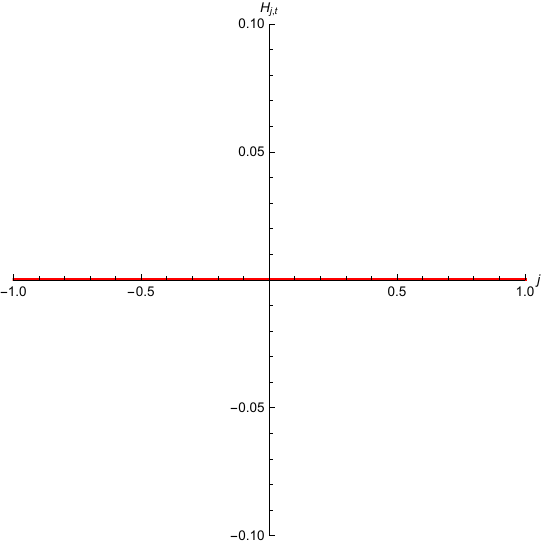}} 
 & \subfloat[$t = 0.3$]{\includegraphics[scale=\scale]{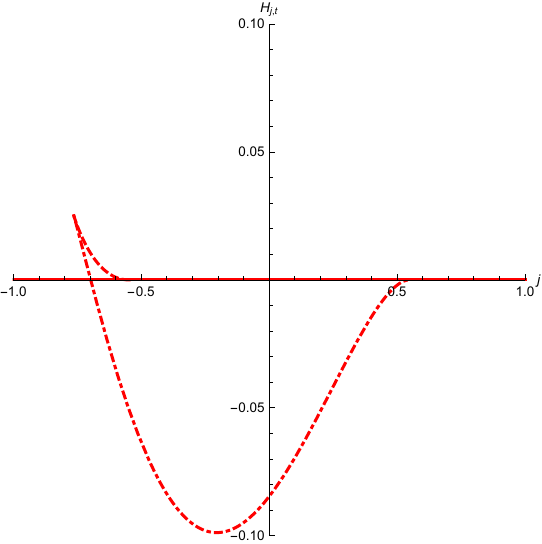}}
\end{tabular}
\caption{Here $a = b = 1$, and $\nu_{1}(j,t) = j^{2}-t$, $\nu_{2}(j,t) = j-t$.}
\label{fig:bif-diag-normal-form-3}

\begin{tabular}{cccc} 
 \subfloat[Degenerate singularities.]{\includegraphics[scale=\scale]{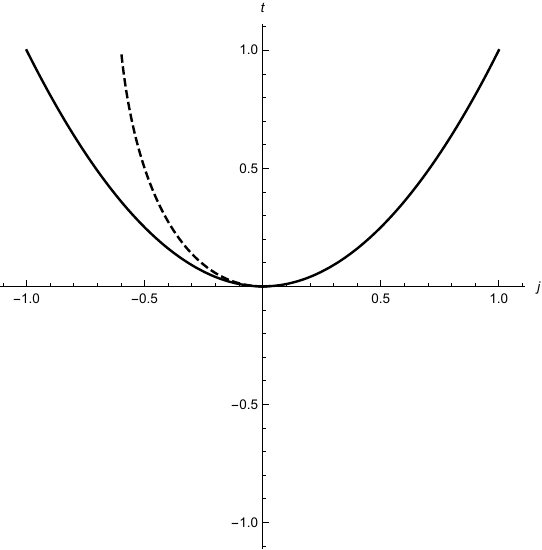}}
 & \subfloat[$t = -0.5$]{\includegraphics[scale=\scale]{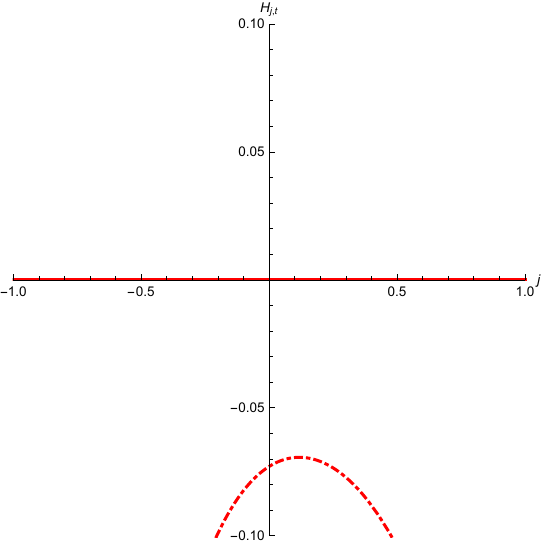}} 
 & \subfloat[$t = 0$]{\includegraphics[scale=\scale]{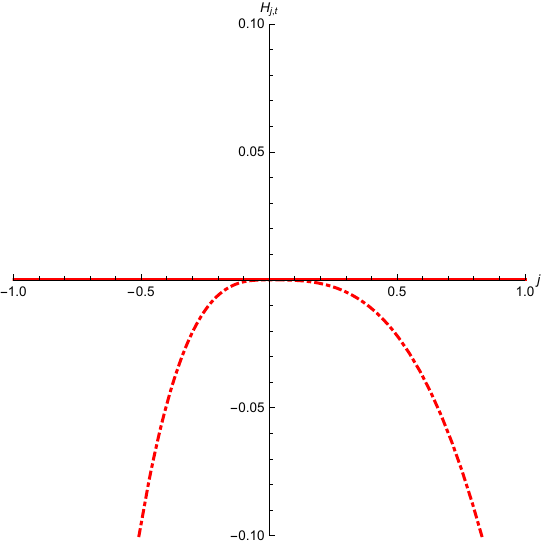}} 
 & \subfloat[$t = 0.5$]{\includegraphics[scale=\scale]{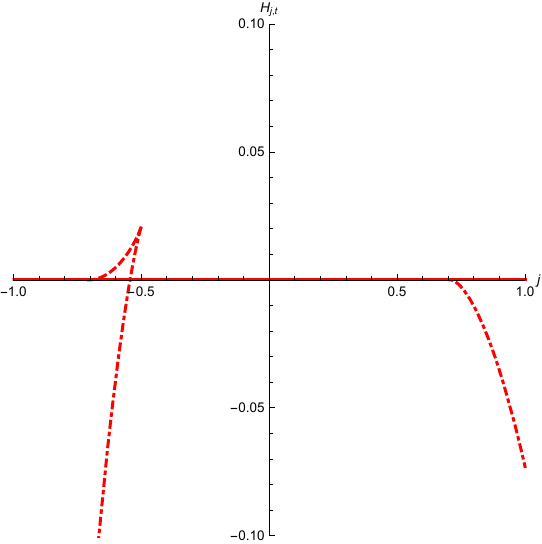}}
\end{tabular}
\caption{Here $a = b = 1$, and $\nu_{1}(j,t) = -j^{2}+t$, $\nu_{2}(j,t) = j-t$.}
\label{fig:bif-diag-normal-form-4}
\end{figure} 

In Theorem \ref{thm:sing-type} we did not use the assumption that $\nu_{1}(j,t)$ describes a saddle-node bifurcation. Let us now make use of this assumption, and study how this affects $\disc_{b}(j,t)$. We compute the following derivatives, and find:
\begin{align*}
&\frac{\partial \disc_{b}}{\partial j}(j,t) = 2\nu_{2}(j,t) \frac{\partial \nu_{2}}{\partial j}(j,t) - 4b \frac{\partial \nu_{1}}{\partial j}(j,t) \\
&\qquad \implies \frac{\partial \disc_{b}}{\partial j}(0,0) = 0, \\
&\frac{\partial^{2} \disc_{b}}{\partial j^{2}}(j,t) = 2 \left( \frac{\partial \nu_{2}}{\partial j}(j,t) \right)^{2} + 2\nu_{2}(j,t) \frac{\partial^{2} \nu_{2}}{\partial j^{2}}(j,t) - 4b \frac{\partial^{2} \nu_{1}}{\partial j^{2}}(j,t) \\
&\qquad \implies \frac{\partial^{2} \disc_{b}}{\partial j^{2}}(0,0) = 2 \left( \frac{\partial \nu_{2}}{\partial j}(0,0) \right)^{2} - 4b \frac{\partial^{2} \nu_{1}}{\partial j^{2}}(0,0), \\
&\frac{\partial \disc_{b}}{\partial t}(j,t) = 2\nu_{2}(j,t) \frac{\partial \nu_{2}}{\partial t}(j,t) - 4b \frac{\partial \nu_{1}}{\partial t}(j,t) \\
&\qquad \implies \frac{\partial \disc_{b}}{\partial t}(0,0) = - 4b \frac{\partial \nu_{1}}{\partial t}(0,0) \neq 0.
\end{align*}
Hence, whenever 
\begin{equation} \label{eq:non-deg}
\frac{\partial^{2} \disc_{b}}{\partial j^{2}}(0,0) \neq 0 
\Leftrightarrow
\left( \frac{\partial \nu_{2}}{\partial j}(0,0) \right)^{2} \neq 2b \frac{\partial^{2} \nu_{1}}{\partial j^{2}}(0,0),
\end{equation}
also $\disc_{b}(j,t)$ satisfies the conditions for a (non-degenerate) saddle-node bifurcation. Therefore we call the double flip bifurcation \emph{non-degenerate} if \eqref{eq:non-deg} is satisfied, and \emph{degenerate} otherwise. An example of a degenerate double flip bifurcation is given in Figure \ref{fig:bif-diag-normal-form-2}.

In the non-degenerate case we can find conditions for when the two saddle-node bifurcations associated with $\nu_{1}(j,t)$ and $\disc_{b}(j,t)$, i.e.\ the curves from figures (a) of Figures \ref{fig:bif-diag-normal-form-1}, \ref{fig:bif-diag-normal-form-3}, and \ref{fig:bif-diag-normal-form-4}, are concave up or down. Let us first recall that, since $\nu_{1}(j,t)$ is assumed to satisfy the conditions for a saddle-node bifurcation, we can write 
\begin{equation*}
\nu_{1}(j,t) = \frac{\partial^{2} \nu_{1}}{\partial j}(0,0) j^{2} + \frac{\partial \nu_{1}}{\partial t}(0,0) t 
\end{equation*}
near $(j,t) = (0,0)$. Thus, the saddle-node is
\begin{itemize}
    \item concave up if $\frac{\partial^{2} \nu_{1}}{\partial j^{2}}(0,0) \frac{\partial \nu_{1}}{\partial t}(0,0) < 0$, and 
    \item concave down if $\frac{\partial^{2} \nu_{1}}{\partial j^{2}}(0,0) \frac{\partial \nu_{1}}{\partial t}(0,0) > 0$.
\end{itemize}
Likewise, we find that the product in the condition for concavity up/down for $\disc_{b}(j,t)$ is
\begin{equation*}
\frac{\partial^{2} \disc_{b}}{\partial j^{2}}(0,0) \frac{\partial \disc_{b}}{\partial t}(0,0) 
= 16b^{2} \frac{\partial^{2} \nu_{1}}{\partial j^{2}}(0,0) \frac{\partial \nu_{1}}{\partial t}(0,0) - 8b \left( \frac{\partial \nu_{2}}{\partial j}(0,0) \right)^{2} \frac{\partial \nu_{1}}{\partial t}(0,0).
\end{equation*}
Thus, we find a criterion for when the concavity of the two curves are the same or opposite, as given in the following theorem:

\begin{theorem} \label{thm:concave}
The saddle-node bifurcations associated to $\nu_{1}(j,t)$ and $\disc_{b}(j,t)$ are concave the same direction if
\begin{equation*}
\frac{1}{2b} \frac{(\partial \nu_{2}/\partial j (0,0))^{2}}{\partial^{2} \nu_{1}/\partial j^{2} (0,0)} < 1.
\end{equation*}
If the inequality is reversed, they are concave in the opposite direction. If it is an equality, the double flip bifurcation is degenerate.
\end{theorem}

\begin{proof}
Assume that the saddle-node bifurcation associated to $\nu_{1}(j,t)$ is concave up, i.e.\ $\frac{\partial^{2} \nu_{1}}{\partial j^{2}}(0,0) \frac{\partial \nu_{1}}{\partial t}(0,0) < 0$. Then
\begin{equation*}
\frac{\partial^{2} \disc_{b}}{\partial j^{2}}(0,0) \frac{\partial \disc_{b}}{\partial t}(0,0) < 0 \quad \implies \quad
\frac{1}{2b} \frac{(\partial \nu_{2}/\partial j (0,0))^{2}}{\partial^{2} \nu_{1}/\partial j^{2} (0,0)} < 1.
\end{equation*}
If we, on the other hand, assume the saddle-node bifurcation associated to $\nu_{1}(j,t)$ is concave down, i.e.\ $\frac{\partial^{2} \nu_{1}}{\partial j^{2}}(0,0) \frac{\partial \nu_{1}}{\partial t}(0,0) > 0$, then we also get
\begin{equation*}
\frac{\partial^{2} \disc_{b}}{\partial j^{2}}(0,0) \frac{\partial \disc_{b}}{\partial t}(0,0) > 0 
\quad \implies \quad
\frac{1}{2b} \frac{(\partial \nu_{2}/\partial j (0,0))^{2}}{\partial^{2} \nu_{1}/\partial j^{2} (0,0)} < 1. \qedhere
\end{equation*}
\end{proof}
\section{Examples} \label{sec:examples}

In this section we consider three examples displaying double flip bifurcations. All three examples are $2$-degree of freedom systems, which we reduce to $1$-degree of freedom, so that we may put the Hamiltonian in the normal form from Theorem \ref{thm:normal-form}. We also employ Theorem \ref{thm:concave} to determine whether the two saddle-node bifurcations are concave in the same direction or not.

\subsection{Oscillator with $(1:-2)$-resonance} \label{sec:Nekhoroshev-oscillator}

First we consider a perturbation of the oscillator-system introduced by Nekhoroshev, Sadovski{\'i} and Zhilinski{\'i} \cite{Nekhoroshev2002}, see also \cite{Nekhoroshev2006}, Efstathiou, Cushman and Sadovski{\'i} \cite{Efstathiou2007}, Martynchuk and Efstathiou \cite{Martynchuk2017}, and Colin de Verdi{\`e}re and \vietnamese{{V{\~u} Ng{\d o}c}} \cite[Example 3.3.2]{colindeverdiere2003} (which is on the $1:2$ resonance). They considered an integrable Hamiltonian system with two degrees of freedom. The underlying symplectic manifold is $(\R^{4},\omega = dq_{1} \wedge dp_{1} + dq_{2} \wedge dp_{2})$, and the two integrals are
\begin{align}\label{eq:oscillator_t=1}
\begin{cases}
J = \frac{1}{2}(q_{1}^{2}+p_{1}^{2}) - (q_{2}^{2}+p_{2}^{2}), \\
H_{\varepsilon,1} = X + \varepsilon R^{2},
\end{cases}
\end{align}
where, $\Im(z)$ being the imaginary part of $z$,
\begin{align*}
&X = \Im((q_{1}+ip_{1})^{2}(q_{2}+ip_{2})) = 2q_{1}p_{1}q_{2} + (q_{1}^{2}-p_{1}^{2})p_{2}, \\
&R = \frac{1}{2}(q_{1}^{2}+p_{1}^{2}) + (q_{2}^{2}+p_{2}^{2}).
\end{align*}
Note that $J$ is the momentum for a $1:-2$ resonant circle action. In particular, in complex coordinates $z = q_{1} + ip_{1}$ and $w = q_{2} + ip_{2}$, the flow corresponding to the Hamiltonian vector field of $J$ is
\begin{equation*}
(z(t),w(t)) = (e^{-it}z(0),e^{2it}w(0)).
\end{equation*}
Thus, points with $z = 0$ and $w \neq 0$ have $\Z_{2}$ as isotropy group. Since the normal form from Theorem \ref{thm:normal-form} is $\Z_{2}$-symmetric, we expect that any point going through the bifurcation satisfies $q_{1} = p_{1} = 0$, and $q_{2} \neq 0$ and/or $p_{2} \neq 0$. Note that these points all correspond to negative values for $J$.

Let us also introduce, denoting by $\Re(z)$ the real part of $z$,
\begin{equation*}
Y = \Re((q_{1}+ip_{1})^{2}(q_{2}+ip_{2})) = (q_{1}^{2}-p_{1}^{2})q_{2} - 2q_{1}p_{1}p_{2}.
\end{equation*}
Note that $(J+R)^{2}(J-R) + 2(X^{2} + Y^{2}) = 0$. Furthermore, since $R+J = q_{1}^{2}+p_{1}^{2} \geq 0$ and $R-J = 2(q_{2}^{2}+p_{2}^{2}) \geq 0$, we have $R \geq \abs{J}$. These equations define the reduced phase space.

Let us perturb the system \eqref{eq:oscillator_t=1} by adding a parameter $t$ to the Hamiltonian, such that we have the following integrals:
\begin{align}\label{eq:oscillator_t}
\begin{cases}
J = \frac{1}{2}(q_{1}^{2}+p_{1}^{2}) - (q_{2}^{2}+p_{2}^{2}), \\
H_{\varepsilon,t} = (1-t)R + t(X + \varepsilon R^{2}).
\end{cases}
\end{align}
Note that when $t=1$ we obtain the system given by \eqref{eq:oscillator_t=1}. 

The singularities of the system are the points where the level sets of the Hamiltonian are tangent to the reduced phase space. Note that the level set $H_{\varepsilon,t}^{-1}(h)$ is given by solving $H_{\varepsilon,t} = h$ for $X$. Let $R = r$, then the solution is
\begin{equation*}
X = X_{\varepsilon,t}(r) := \frac{1}{t}(h-(1-t)r) - \varepsilon r^{2}.
\end{equation*}
Since it is independent of $Y$, and since $(J+R)^{2}(J-R) + 2(X^{2} + Y^{2}) = 0$ is a surface of revolution, all singularities satisfy $Y = 0$. Hence, the reduced phase space has been drawn in Figure \ref{fig:reduced_phase_space} projected onto the $(Y=0)$-plane, for certain values $J = j$ (see also Efstathiou, Cushman and Sadovski{\'i} \cite[Lemma 17]{Efstathiou2007}).

\begin{figure}[tb]
\def\scale{0.5}
\centering
\begin{tabular}{ccc} 
 {\includegraphics[scale=\scale]{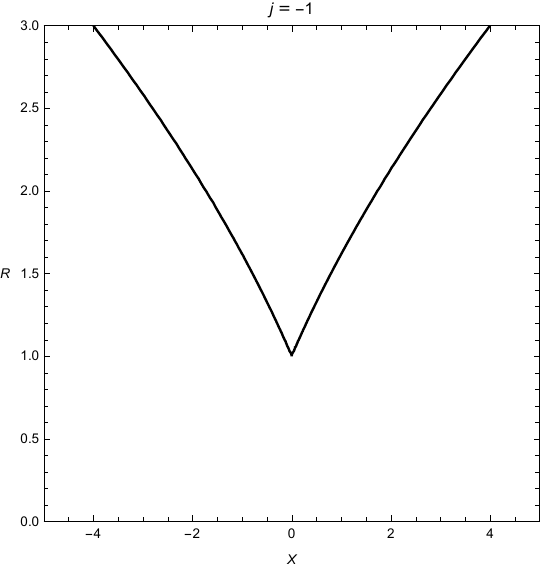}} & {\includegraphics[scale=\scale]{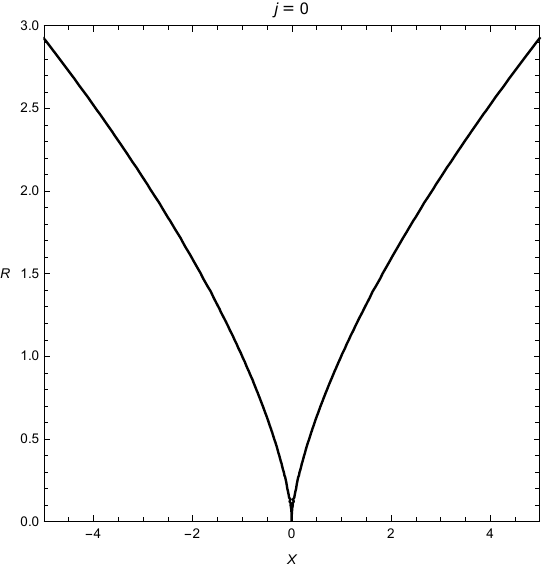}} & {\includegraphics[scale=\scale]{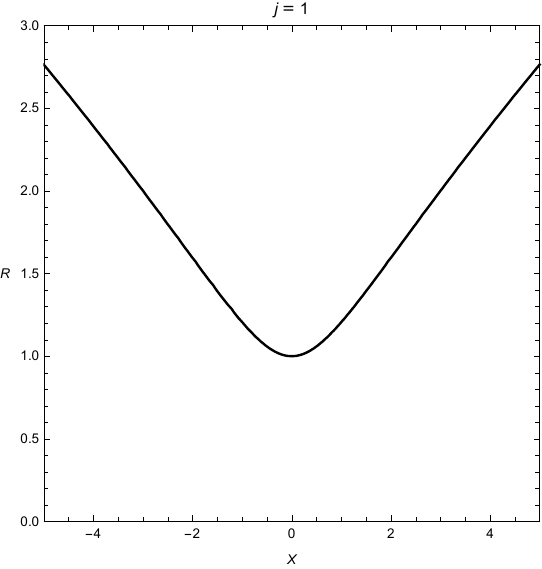}} 
\end{tabular}
\caption{Projections of the reduced phase space onto the $Y=0$ plane. To the left it has a conical singularity at the bottom, in the middle a cuspidal singularity, and to the right no singularity.}
\label{fig:reduced_phase_space}
\end{figure}



Let us now show that the system \eqref{eq:oscillator_t} goes through the double flip bifurcation. Recall that to determine that the system goes through this bifurcation, we put the Hamiltonian in a normal form, evaluated at the point of bifurcation. Thus, we need to find some candidate values for $j$ and $t$, where the system satisfies at least some of the requirements for the bifurcation, and which are preferably not difficult to verify. Hence, we call $(j,t)$ a \emph{candidate} if the level set $H_{\varepsilon,t}^{-1}(h)$ touches the conical singularity on the reduced phase space. Recall that we already can assume that $j$ is negative, as all points having $\Z_{2}$ as isotropy group satify this condition. Furthermore, the conical singularity takes the value $R = -j$. Furthermore, the slope of the level set $H_{\varepsilon,t}^{-1}(h)$, i.e.\ $\frac{\partial X_{\varepsilon,t}}{\partial r}$, is independent of $h$. 

For convenience, let us define
\begin{equation*}
j_{0}^{\pm}(\varepsilon,t) := \frac{4\varepsilon - t(4\varepsilon+1) \pm \sqrt{t} \sqrt{t(8\varepsilon+1) - 8\varepsilon}}{8t\varepsilon^{2}}.
\end{equation*}
Note that $j_{0}^{-}(\varepsilon,t) = j_{0}^{+}(\varepsilon,t) = j_{0}(\varepsilon)$ when $t = t_{0}(\varepsilon)$, where
\begin{equation*}
j_{0}(\varepsilon) := -\frac{1}{16\varepsilon^{2}}, 
\quad
t_{0}(\varepsilon) := \frac{8\varepsilon}{8\varepsilon+1}.
\end{equation*}

\begin{proposition} \label{prop:Nekhoroshev-flip}
The system \eqref{eq:oscillator_t} has a candidate for the double flip bifurcation at $(j,t) = (j_{0}(\varepsilon),t_{0}(\varepsilon))$.
\end{proposition}

\begin{proof}
Let us use the equation for the reduced phase space (projected to the $Y=0$ plane), and obtain $X(j,r) = \pm \sqrt{-(j+r)^{2}(j-r)/2}$. The slope at the singularity $r = -j$ is
\begin{equation*}
\left. \frac{\partial X}{\partial r} \right|_{r=-j} = \mp \sqrt{-j}.
\end{equation*}
Furthermore, the level set $H_{\varepsilon,t} = h$ yields
\begin{equation*}
\left. \frac{\partial X_{\varepsilon,t}}{\partial r} \right|_{r=-j} = 1 - \frac{1}{t} + 2\varepsilon j.
\end{equation*}
Solving $\frac{\partial X}{\partial r}|_{r=-j} = \frac{\partial X_{\varepsilon,t}}{\partial r}|_{r=-j}$ for $j$ yields $j = j_{0}^{\pm}(\varepsilon,t)$. Since $j$ has to be a real number, the Hamiltonian flip bifurcation takes place only when the discriminant is non-negative, i.e.\ when $t(1+8\varepsilon) - 8\varepsilon \geq 0$, or, equivalently, when $t \geq \frac{8\varepsilon}{8\varepsilon + 1}$. In particular, the candidate is given by the values where the discriminant vanishes, i.e.\ $(j,t) = (j_{0}(\varepsilon),t_{0}(\varepsilon))$.
\end{proof}

The normal form found in Theorem \ref{thm:normal-form} is that for a $1$-degree of freedom system, so we reduce the Hamiltonian $H_{\varepsilon,t}$ by the other integral $J$. Since $J$ defines an $S^{1}$-action, we may set $p_{2} = 0$, and write
\begin{equation*}
q_{2} = \pm \frac{1}{\sqrt{2}} \sqrt{q_{1}^{2} + p_{1}^{2} - 2j}.
\end{equation*}
Note also that 
\begin{equation*}
R + J = q_{1}^{2}+p_{1}^{2} \implies R = q_{1}^{2}+p_{1}^{2}-J.
\end{equation*}
Then, for some value $J = j$, the reduced Hamiltonian is
\begin{equation*}
H_{\varepsilon,j,t}^{\textup{red}} =
(1-t)(q_{1}^{2}+p_{1}^{2}-j) + t \left( \pm\sqrt{2}q_{1}p_{1}\sqrt{q_{1}^{2}+p_{1}^{2}-2j} + \varepsilon(q_{1}^{2}+p_{1}^{2}-j)^{2} \right),
\end{equation*}
where the $\pm$-sign comes from the choice of $q_{2}$. This does not matter here, as we anyway evaluate at the point $q_{1} = p_{1} = 0$, where it vanishes.

To simplify the notation, let us abbreviate $q = q_{1}$ and $p = p_{1}$. Let us compute the Taylor expansion of $H_{\varepsilon,j,t}^{\textup{red}}$ at $(q,p) = (0,0)$ up to sixth order, where $(H_{\varepsilon,j,t}^{\text{red}})_{i}$ is the $i$th order term:
\begin{align*}
&(H_{\varepsilon,j,t}^{\text{red}})_{0}(q,p) = \frac{3}{32\varepsilon(1+8\varepsilon)}, \\
&(H_{\varepsilon,j,t}^{\text{red}})_{2}(q,p) = \frac{2}{1+8\varepsilon}(p + q)^{2}, \\
&(H_{\varepsilon,j,t}^{\text{red}})_{4}(q,p) = \frac{8\varepsilon^{2}}{1+8\varepsilon}(p^{2} + q^{2})(p + q)^{2}, \\
&(H_{\varepsilon,j,t}^{\text{red}})_{6}(q,p) = -\frac{32\varepsilon^{4}}{1+8\varepsilon}pq(p^{2} + q^{2})^{2},
\end{align*}
and the odd-degree terms vanish. To put this in normal form, we first need to do something about the $2$nd degree term. To get rid of the $pq$-term, we consider the following coordinate transformation:
\begin{equation*}
\underbrace{\frac{1}{\sqrt{2}}\begin{pmatrix}
1 & 1 \\
-1 & 1
\end{pmatrix}}_{=:A}
\begin{pmatrix}
q \\
p
\end{pmatrix}
=
\frac{1}{\sqrt{2}}
\begin{pmatrix}
q + p \\
-q + p
\end{pmatrix}.
\end{equation*}
Note that this transformation preserves the symplectic structure:
\begin{equation*}
A^{T}\Omega_{0}A =
\frac{1}{2}
\begin{pmatrix}
1 & -1 \\
1 & 1
\end{pmatrix}
\begin{pmatrix}
0 & 1 \\
-1 & 0
\end{pmatrix}
\begin{pmatrix}
1 & 1 \\
-1 & 1
\end{pmatrix}
= 
\begin{pmatrix}
0 & 1 \\
-1 & 0
\end{pmatrix}.
\end{equation*}
Under this coordinate transformation, the reduced Hamiltonian takes the following form:
\begin{equation*}
\tilde{H}_{\varepsilon,j,t}^{\text{red}}(q,p) = (1-t)(q^{2}+p^{2}-j) + t \left( \pm \frac{-q^{2}+p^{2}}{\sqrt{2}} \sqrt{q^{2}+p^{2}-2j} + \varepsilon(q^{2}+p^{2}-j)^{2} \right).
\end{equation*}
Now the first six terms of the Taylor expansion at the origin are, in terms of coordinates $u = \frac{1}{2}q^{2}$, $v = \frac{1}{2}p^{2}$, and $w = qp$,
\begin{align*}
&(\tilde{H}_{\varepsilon,j,t}^{\text{red}})_{0} = j(jt\varepsilon+t-1), \\
&(\tilde{H}_{\varepsilon,j,t}^{\text{red}})_{2} = u \left( 2(1-t) - 2t(\sqrt{-j}+2j\varepsilon) \right) + v \left( 2(1-t) + 2t(\sqrt{-j}-2j\varepsilon) \right), \\
&(\tilde{H}_{\varepsilon,j,t}^{\text{red}})_{4} = u^{2} \left( 4\varepsilon - \frac{1}{\sqrt{-j}} \right) + 2w^{2}t\varepsilon + v^{2} \left( 4\varepsilon + \frac{1}{\sqrt{-j}} \right), \\
&(\tilde{H}_{\varepsilon,j,t}^{\text{red}})_{6} = \frac{t}{16\sqrt{-j^{3}}} \left(4u^{3} + uw^{2} - vw^{2} + 4v^{3} \right).
\end{align*}
Again, the odd-degree terms all vanish. Denote by
\begin{equation*}
\nu_{1}(\varepsilon,j,t) = 2(1-t) - 2t(\sqrt{-j}+2j\varepsilon)
\end{equation*}
the coefficient of $u$ in $(\tilde{H}_{\varepsilon,j,t}^{\text{red}})_{2}$. This vanishes at $(j,t) = (j_{0}(\varepsilon),t_{0}(\varepsilon))$. Furthermore,
\begin{align*}
&\frac{\partial \nu_{1}}{\partial j}(\varepsilon,j_{0}(\varepsilon),t_{0}(\varepsilon)) = 0, 
\quad
\frac{\partial^{2} \nu_{1}}{\partial j^{2}}(\varepsilon,j_{0}(\varepsilon),t_{0}(\varepsilon)) = \frac{256\varepsilon^{4}}{8\varepsilon+1}, \\
&\frac{\partial \nu_{1}}{\partial t}(\varepsilon,j_{0}(\varepsilon),t_{0}(\varepsilon)) = -\frac{8\varepsilon+1}{4\varepsilon}.
\end{align*}
Note also that, for all negative $j$, the coefficient of $v$ in $(\tilde{H}_{\varepsilon,j,t}^{\text{red}})_{2}$, denoted below by $a$, is positive for all $t \in [0,1]$. In particular, $a$ is positive for all $t \in [0,1]$ in a neighbourhood of $j_{0}(\varepsilon)$. Thus, the second order terms are of the desired form. Evaluated at $(j_{0}(\varepsilon),t_{0}(\varepsilon))$ we have
\begin{equation*}
a = \frac{8}{8\varepsilon+1}.
\end{equation*}

The higher order terms require more work. For this, let us do near identity transformations by introducing a function $F$, and consider:
\begin{equation*}
e^{\textup{ad}_F} = I + \ad_{F} + \frac{1}{2}\ad_{F}\ad_{F} + \cdots,
\end{equation*}
where $\ad_{F} = \{F,\cdot\}$, and $\{\cdot,\cdot\}$ is the Poisson bracket associated with the symplectic form. Now, let $F = F_{2} = f_{1}u^{2} + f_{2}v^{2} + f_{3}w^{2} + f_{4}uw + f_{5}vw$ be a degree-$2$ homogeneous polynomial with coordinates $u,v,w$. The coefficients $f_{i}$, $i \in \{1,...,5\}$, depend on the parameters $\varepsilon$, $j$ and $t$. Applying the transformation to the Taylor expansion of $\tilde{H}_{\varepsilon,j,t}^{\text{red}}$ we get
\begin{align*}
e^{\ad_{F_{2}}}\tilde{H}_{\varepsilon,j,t}^{\text{red}}
= &(\tilde{H}_{\varepsilon,j,t}^{\text{red}})_{0} + (\tilde{H}_{\varepsilon,j,t}^{\text{red}})_{2} + \left( (\tilde{H}_{\varepsilon,j,t}^{\text{red}})_{4} + \ad_{F_{2}}(\tilde{H}_{\varepsilon,j,t}^{\text{red}})_{2} \right) \\&+ \left( (\tilde{H}_{\varepsilon,j,t}^{\text{red}})_{6} + \ad_{F_{2}}(\tilde{H}_{\varepsilon,j,t}^{\text{red}})_{4} + \frac{1}{2}\ad_{F_{2}}\ad_{F_{2}}(\tilde{H}_{\varepsilon,j,t}^{\text{red}})_{2} \right) + \cdots.
\end{align*}
The brackets contain terms of the same degree. In the normal form, the only term of degree $2$ is $u^{2}$. Hence we find a $F_{2}$ such that $(\tilde{H}_{\varepsilon,j,t}^{\text{red}})_{4} + \ad_{F_{2}}(\tilde{H}_{\varepsilon,j,t}^{\text{red}})_{2}$ yields only $u^{2}$-terms. Then we find the second function of the normal form, the coefficient of $u^{2}$:
\begin{equation*}
\nu_{2}(\varepsilon,j,t) = \frac{4t \left( 8\varepsilon - t(3+16\varepsilon+32\varepsilon^{2}j) + t^{2}(3+2\varepsilon(4+j) + 32\varepsilon^{2}j + 32\varepsilon^{3}j^{2} \right)}{3(1 - t(1-\sqrt{-j}+2\varepsilon j))^{2}}.
\end{equation*}
Note that
\begin{equation*}
\nu_{2}(\varepsilon,j_{0}(\varepsilon),t_{0}(\varepsilon)) = 0, 
\quad \text{and} \quad
\frac{\partial \nu_{2}}{\partial j}(\varepsilon,j_{0}(\varepsilon),t_{0}(\varepsilon)) = - \frac{256\varepsilon^{4}}{8\varepsilon+1}.
\end{equation*}
Similarly we show that for the cubic terms, we can find some degree-$3$ homogeneous polynomial $F_{3}$ such that the only degree $3$ term surviving is the $u^{3}$-term. Note that this does not influence the lower order terms. The coefficient of $u^{3}$, evaluated at $(j,t) = (j_{0}(\varepsilon),t_{0}(\varepsilon))$, is
\begin{equation*}
b = \frac{96\varepsilon^{4}}{8\varepsilon+1}.
\end{equation*}

We want to employ Theorem \ref{thm:normal-form} to determine the type of the Hamiltonian flip bifurcations along $j_{0}^{\pm}(\varepsilon,t)$. Thus, we need to find the sign of $\nu_{2}(\varepsilon,j_{0}^{\pm}(\varepsilon,t),t)$, for $t > t_{0}(\varepsilon)$. Let
\begin{equation*}
r(\varepsilon,t) := \sqrt{t(4\varepsilon+1) - 4\varepsilon \pm \sqrt{t}\sqrt{t(8\varepsilon+1)-8\varepsilon}}.
\end{equation*}
One can show that
\begin{equation*}
\nu_{2}(\varepsilon,j_{0}^{\pm}(\varepsilon,t),t) = \frac{16\varepsilon t \left( t(8\varepsilon+1) - 8\varepsilon \pm \sqrt{t}\sqrt{t(8\varepsilon+1)-8\varepsilon} \right)}{\left( \sqrt{t(8\varepsilon+1)-8\varepsilon} \pm (\sqrt{t} + \sqrt{2}r(\varepsilon,t)) \right)^{2}}.
\end{equation*}
The denominator is squared, so it is positive as long as its root is real. Let us first verify this. The only part of the root needing to be checked is $r(\varepsilon,t)$. Let us substitute $t = t_{0}(\varepsilon) + \frac{\delta}{8\varepsilon+1}$, for some positive quantity $\delta > 0$, where the factor $\frac{1}{8\varepsilon+1}$ in front of $\delta$ has been introduced to simplify the calculation. Then the radicand of $r(\varepsilon,t)$ can be written as
\begin{equation*}
\frac{\delta(4\varepsilon+1) + 4\varepsilon \pm \sqrt{\delta(8\varepsilon+\delta)(8\varepsilon+1)}}{8\varepsilon+1}.
\end{equation*}
Since each term is positive, we only need to consider the case with the minus-sign. Furthermore, since we are only interested in the sign of the expression, we may simply compare the square of each term of the numerator:
\begin{equation*}
(\delta(4\varepsilon+1) + 4\varepsilon)^{2} - \delta(8\varepsilon+\delta)(8\varepsilon+1)
=
16\varepsilon^{2}(\delta-1)^{2} 
\geq 
0.
\end{equation*}
Hence, the denominator of $\nu_{2}(\varepsilon,j_{0}^{\pm}(\varepsilon,t),t)$ is positive for all $t > t_{0}(\varepsilon)$. The sign of the numerator, on the other hand, is positive for $\nu_{2}(\varepsilon,j_{0}^{+}(\varepsilon,t),t)$, and negative for $\nu_{2}(\varepsilon,j_{0}^{-}(\varepsilon,t),t)$. Hence, since $a > 0$, we have Hamiltonian flip bifurcations along $j_{0}^{-}(\varepsilon,t)$, and dual Hamiltonian flip bifurcations along $j_{0}^{+}(\varepsilon,t)$.

Finally, we compute
\begin{equation*}
\frac{1}{2b} \frac{\left( \partial \nu_{2} / \partial j (\varepsilon,j_{0}(\varepsilon),t_{0}(\varepsilon)) \right)}{\partial^{2} \nu_{1} / \partial j^{2} (\varepsilon,j_{0}(\varepsilon),t_{0}(\varepsilon))} = \frac{4}{3}.
\end{equation*}
Thus, by Theorem \ref{thm:concave} the bifurcation is non-degenerate with oppositely concave parabolas. 

Figure \ref{fig:osc-bif-diag} shows bifurcation diagrams for the system given by \eqref{eq:oscillator_t} for selected values of $t$. Here we see a cusp colliding with a curve of elliptic-regular values, producing two Hamiltonian-flip bifurcations. This is similar to Figure \ref{fig:bif-diag-normal-form-1}. See also Figure \ref{fig:osc-concavity}, where one can see the two saddle-node bifurcations, and that they are oppositely concave. Notice also that in the middle figure there are two overlapping curves, for $j \in [j_{0}(\varepsilon),0]$. In fact, this is also the case in Example \ref{sec:ex_Hirzebruch}, see Figure \ref{fig:W2-bif-diag}, but not in the following example, see Figure \ref{fig:osc2-bif-diag}.

\begin{figure}[tb]
\def\scale{0.55}
\centering
\begin{tabular}{ccc} 
 {\includegraphics[scale=\scale]{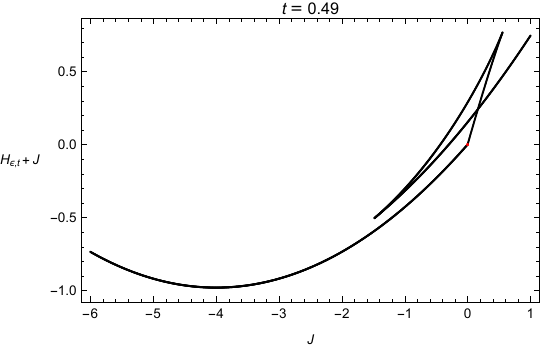}} & 
 {\includegraphics[scale=\scale]{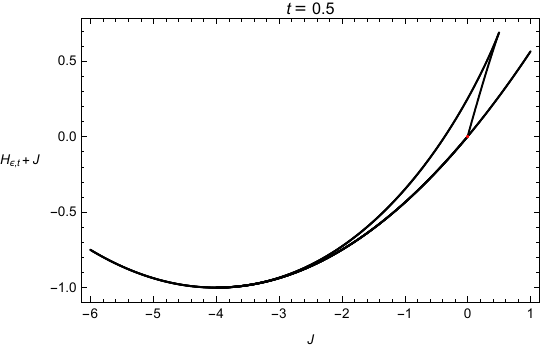}} & 
 {\includegraphics[scale=\scale]{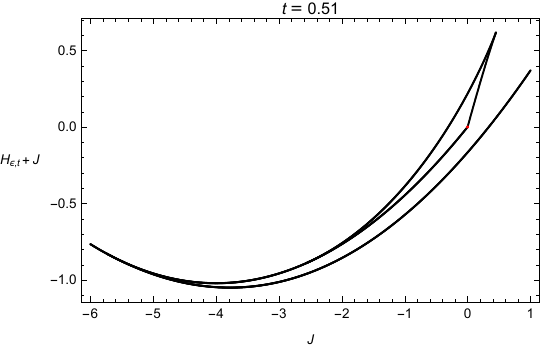}}
\end{tabular}
\caption{Bifurcation diagrams of the system given by \eqref{eq:oscillator_t} for selected values of $t$, with $\varepsilon = \frac{1}{8}$. Here we use $H_{\varepsilon,t}+J$ instead of $H_{\varepsilon,t}$ to more easily see the various branching points. The Hamiltonian flip bifurcations occur for $t \geq \frac{1}{2}$.}
\label{fig:osc-bif-diag}
\end{figure}

\begin{figure}[tb]
\centering
\includegraphics[width=0.8\linewidth]{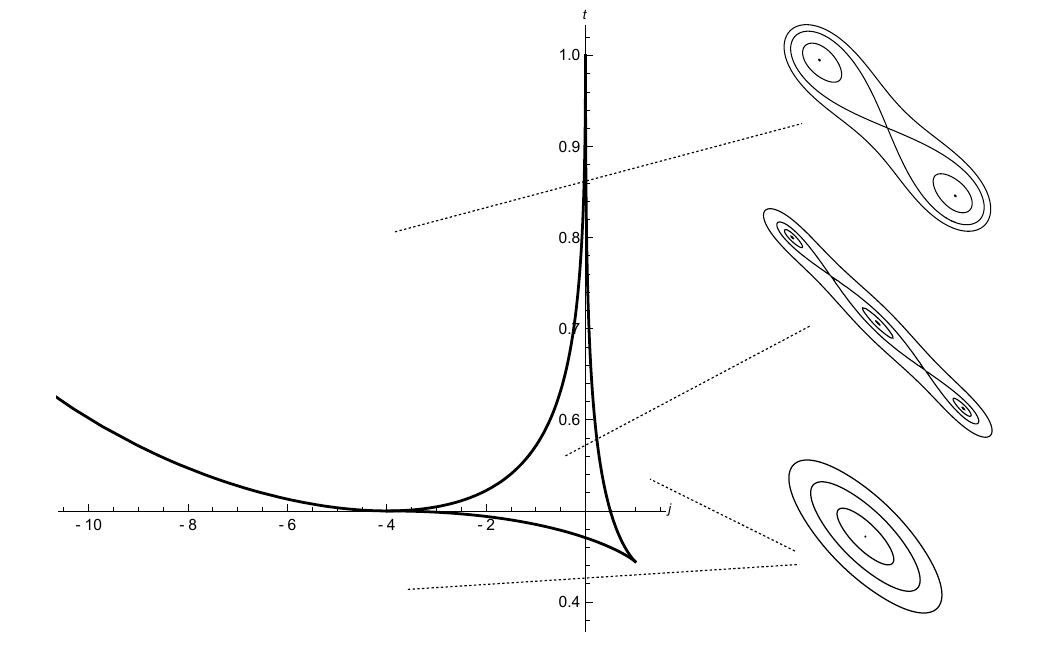}
\caption{The curves display the degenerate singularities, which bound the areas of different topological level sets.}
\label{fig:osc-concavity}
\end{figure}

\subsection{A slight modification of the previous oscillator}

We modify now the system \eqref{eq:oscillator_t} by increasing the exponent of $\varepsilon R^{2}$ by $1$, i.e.\ we consider
\begin{align} \label{eq:osc2}
\begin{cases}
J = \frac{1}{2}(q_{1}^{2}+p_{1}^{2}) - (q_{2}^{2}+p_{2}^{2}), \\
H_{\varepsilon,t} = (1-t)R + t(X + \varepsilon R^{3}).
\end{cases}
\end{align}
Otherwise everything else remains unchanged. However, as this turns out to be much more computationally heavy, let us fix a value for $\varepsilon$. We choose $\varepsilon = \frac{9}{256}$, which implies that $(j,t) = (-\frac{16}{9},\frac{1}{2})$ is a candidate for the double flip bifurcation. Then one computes the normal form in exactly the same way as in the previous example, showing that a double flip bifurcation occurs in the system. 

This example differs from the previous one in the sense that there are no overlapping lines in the bifurcation diagram, which is drawn in Figure \ref{fig:osc2-bif-diag}.

\begin{figure}[tb]
\def\scale{0.55}
\centering
\begin{tabular}{ccc} 
 {\includegraphics[scale=\scale]{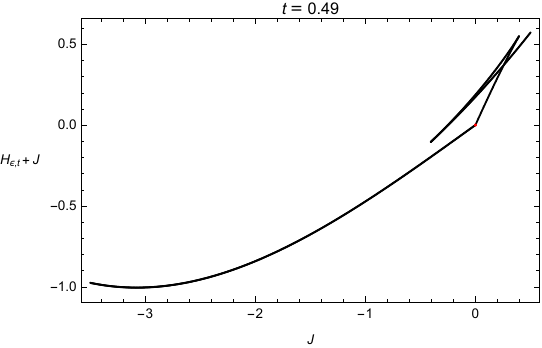}} & 
 {\includegraphics[scale=\scale]{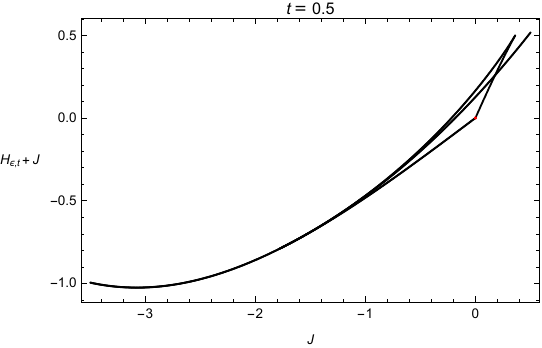}} & 
 {\includegraphics[scale=\scale]{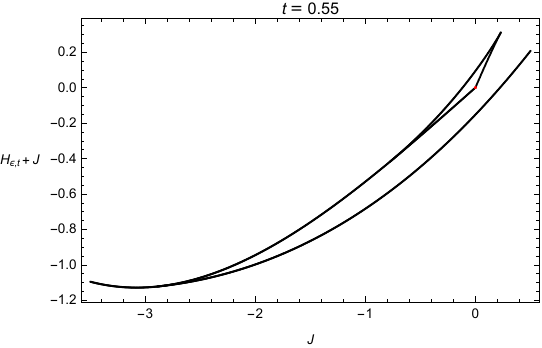}}
\end{tabular}
\caption{Bifurcation diagrams of the system given by \eqref{eq:osc2} for selected values of $t$, with $\varepsilon = \frac{9}{256}$. Here we use $H_{\varepsilon,t}+J$ instead of $H_{\varepsilon,t}$ to more easily see the various branching points. The Hamiltonian flip bifurcations occur for $t \geq \frac{1}{2}$.}
\label{fig:osc2-bif-diag}
\end{figure}

\subsection{Hirzebruch surface} \label{sec:ex_Hirzebruch}

The third example is given by the Hirzebruch surface defined as the symplectic reduction of $(\C^{4},-\frac{i}{2}\sum_{j=1}^{4} dz_{j} \wedge d\bar{z}_{j})$ by the Hamiltonian $T^{2}$-action generated by
\begin{equation*}
N(z_{1},z_{2},z_{3},z_{4}) = -\frac{1}{2} \left( \abs{z_{1}}^{2} + \abs{z_{2}}^{2} + 2\abs{z_{3}}^{2} + 3, \abs{z_{3}}^{2} + \abs{z_{4}}^{2} + 1 \right).
\end{equation*}
This specific Hirzebruch surface is denoted by $W_{2}(1,1)$, see Le Floch and Palmer \cite[Section 7]{LeFloch2018}, Flamand \cite[Example 3.3.7]{Flamand2024}, and Efstathiou, Hohloch and Santos \cite[Section 7.3]{Efstathiou2024}. The integrals of motion that we consider are given by
\begin{align} \label{eq:Hirz-sys}
\begin{cases}
J = \frac{1}{2}\abs{z_{2}}^{2}, \\
H_{t} = (1-t)R + tX,
\end{cases}
\end{align}
where 
\begin{align*}
R = \frac{1}{2}\abs{z_{3}}^{2}, \quad X = \Re(\bar{z}_{1}^{2}z_{3}\bar{z}_{4}).
\end{align*}
As above, we also consider
\begin{equation*}
Y = \Im(\bar{z}_{1}^{2}z_{3}\bar{z}_{4}),
\end{equation*}
which satisfy the syzygy
\begin{equation} \label{eq:Hirz-syzygy}
X^{2} + Y^{2} = 16R(1-R)(3-J-2R)^{2}
\end{equation}
(see Flamand \cite[Lemma 2.5.2]{Flamand2024}). 

The level set of the Hamiltonian $H_{t}^{-1}(h)$ is given by
\begin{equation*}
X_{t}(r) = \frac{1}{t}(h-(1-t)r),
\end{equation*}
which is independent of $Y$. Since the syzygy \eqref{eq:Hirz-syzygy} is a surface of revolution, we again have singularities only for $Y = 0$. Hence, we again draw the reduced phase space projected onto the plane given by $Y = 0$, in Figure \ref{fig:Hirz-ps}. We can see a conical singularity for $j > 1$.

\begin{figure}[tb]
\def\scale{0.5}
\centering
\begin{tabular}{ccc} 
 {\includegraphics[scale=\scale]{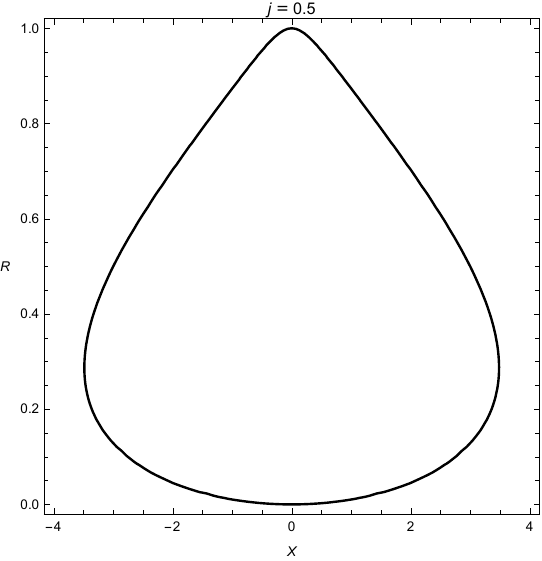}} & {\includegraphics[scale=\scale]{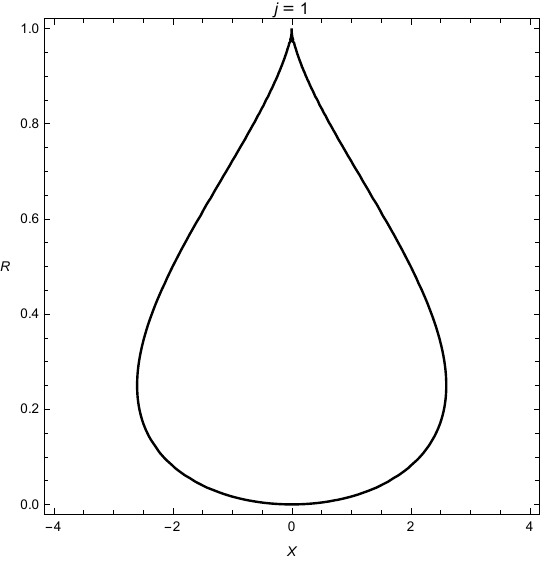}} & {\includegraphics[scale=\scale]{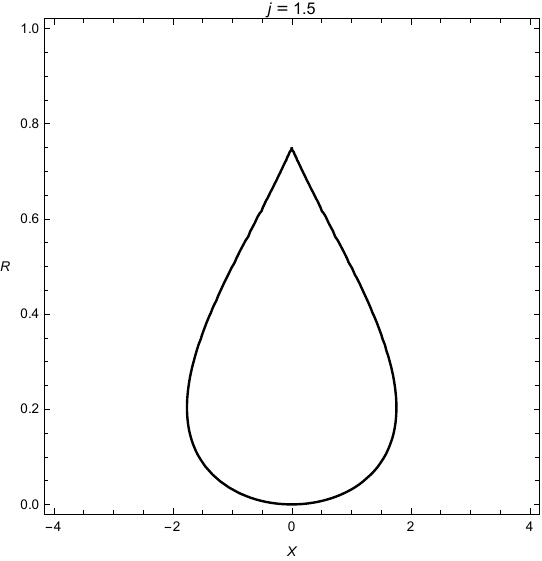}} 
\end{tabular}
\caption{Projections of the reduced phase space onto the $Y=0$ plane.}
\label{fig:Hirz-ps}
\end{figure}

Let us now recall the local coordinates introduced by Le Floch and Palmer \cite[Section 5.2]{LeFloch2018}. Let $z_{j} = x_{j} + iy_{j}$, where $x_{j},y_{j} \in \R$, for $j \in \{1,2,3,4\}$, and let $[z_{1},z_{2},z_{3},z_{4}] \in W_{2}(1,1)$ be the equivalence class of $(z_{1},z_{2},z_{3},z_{4}) \in \C^{4}$ under the action by $N$. Then one can cover $W_{2}(1,1)$ by four charts, defined by
\begin{equation*}
U_{l,m} = \{ [z_{1},z_{2},z_{3},z_{4}] : z_{l} \neq 0 \neq z_{m} \}.
\end{equation*}
Then $W_{2}(1,1) \subset U_{1,3} \cup U_{1,4} \cup U_{2,3} \cup U_{2,4}$. Furthermore, one can obtain local coordinates on $U_{l,m}$ such that $z_{l},z_{m} \in \R$, i.e.\ $y_{l} = y_{m} = 0$, and, for
\begin{itemize}
    \item $(l,m) = (1,3)$: $x_{1} = \sqrt{2+2(x_{4}^{2}+y_{4}^{2})-(x_{2}^{2}+y_{2})^{2}}$, $x_{3} = \sqrt{2-(x_{4}^{2}+y_{4}^{2})}$,
    \item $(l,m) = (1,4)$: $x_{1} = \sqrt{6-2(x_{3}^{2}+y_{3}^{2})-(x_{2}^{2}+y_{2}^{2})}$, $x_{4} = \sqrt{2-(x_{3}^{2}+y_{3}^{2})}$,
    \item $(l,m) = (2,3)$: $x_{2} = \sqrt{2+2(x_{4}^{2}+y_{4}^{2})-(x_{1}^{2}+y_{1}^{2})}$, $x_{3} = \sqrt{2-(x_{4}^{2}+y_{4}^{2})}$,
    \item $(l,m) = (2,4)$: $x_{2} = \sqrt{6-2(x_{3}^{2}+y_{3}^{2})-(x_{1}^{2}+y_{1}^{2})}$, $x_{4} = \sqrt{2-(x_{3}^{2}+y_{3}^{2})}$.
\end{itemize}

We consider the integral $J$ with the local coordinates on the chart $U_{2,3}$:
\begin{equation*}
J(x_{1},y_{1},x_{4},y_{4}) = -\frac{1}{2}(x_{1}^{2}+y_{1}^{2}) + (x_{4}^{2}+y_{4}^{2}).
\end{equation*}
Hence, $J$ is in the normal form for a $1:-2$ resonant system. Note that the points with $x_{1} = y_{1} = 0$ have $\Z_{2}$ as isotropy group. 

We may use the local coordinates on $U_{2,3}$ to find a candidate for the double flip bifurcation. A candidate can also be the value $t$ where one first encounters degenerate singularities on the points have $\Z_{2}$ as isotropy group.

\begin{proposition}
The system \eqref{eq:Hirz-sys} has a candidate for the double flip bifurcation at $(j,t) = (j_{0},t_{0}) := (2,\frac{1}{5})$.
\end{proposition}

\begin{proof}
We reduce the Hamiltonian, and compute the eigenvalues of the linearised Hamiltonian vector field. Then we find for what values of $j$ and $t$ the eigenvalues vanish.

With the local coordinates on $U_{2,3}$ we solve $J = j$ to find
\begin{equation*}
x_{4} = \sqrt{j-1+\frac{1}{2}(x_{1}^{2}+y_{1}^{2})},
\end{equation*}
where we have used that $J$ defines an $S^{1}$-action to eliminate $y_{4}$. Then we find the reduced Hamiltonian
\begin{align*}
H_{j,t}^{\textup{red}}(x_{1},y_{1}) = \frac{1}{4} \bigg( &
6 + 2j(t-1) - 6t + (x_{1}^{2}+y_{1}^{2})(t-1) + \\&(x_{1}^{2}-y_{1}^{2})t\sqrt{(6-2j-x_{1}^{2}-y_{1}^{2})(2j-2+x_{1}^{2}+y_{1}^{2})}
\bigg).
\end{align*}
Computing the Hamiltonian vector field of $H_{j,t}$, and evaluating at $(x_{1},y_{1}) = (0,0)$, yields the following eigenvalues:
\begin{equation*}
\lambda^{\pm} = \pm \frac{1}{2} \sqrt{-16j^{2}t^{2} + 64jt^{2} - 49t^{2} + 2t - 1}.
\end{equation*}
The eigenvalue $\lambda^{\pm}$ vanishes if $(j,t) = (2,\frac{1}{5})$, or if
\begin{equation*}
t > \frac{1}{5}, 
\quad
j = \frac{8t \pm \sqrt{(5t-1)(3t+1)}}{4t}.
\end{equation*}
Thus, the natural candidate is $(j,t) = (2,\frac{1}{5})$.
\end{proof}

A similar procedure to the one given for the oscillator may be used to normalise the Hamiltonian for the Hirzebruch surface. The normal form is
\begin{align*}
\tilde{H}_{j,t}^{\textup{red}}(x_{1},y_{1})
= 
&\frac{(t-1)(j-3)}{2}
- \frac{1 - t + 4t\sqrt{-3+4j-j^{2}}}{4}y_{1}^{2} \\&
- \frac{1 - t - 4t\sqrt{-3+4j-j^{2}}}{4}x_{1}^{2} \\&
+ \frac{8t^{2}(t-1)(j-2)}{(1 - t(1-4\sqrt{-3+4j-j^{2}}))^{2}}x_{1}^{4} \\&
+ \frac{4t^{2}(t-1)}{((j-3)(j-1)(1 - t(1-4\sqrt{-3+4j-j^{2}}))^{3})}x_{1}^{6}.
\end{align*}
We readily read off from this expression
\begin{align*}
&\nu_{1}(j,t) = - \frac{1 - t - 4t\sqrt{-3+4j-j^{2}}}{2}, \\
&\nu_{2}(j,t) = \frac{32(t-1)(j-2)}{(1 - t(1-4\sqrt{-3+4j-j^{2}}))^{2}}.
\end{align*}
Note that 
\begin{align*}
&\nu_{1}(j_{0},t_{0}) = 0,
\quad
\frac{\partial \nu_{1}}{\partial j}(j_{0},t_{0}) = 0,
\quad
\frac{\partial^{2} \nu_{1}}{\partial j^{2}}(j_{0},t_{0}) = -\frac{2}{5},
\quad
\frac{\partial \nu_{1}}{\partial t}(j_{0},t_{0}) = \frac{2}{5}, \\
&\nu_{2}(j_{0},t_{0}) = 0,
\quad
\frac{\partial \nu_{2}}{\partial j}(j_{0},t_{0}) = -\frac{2}{5}.
\end{align*}
Furthermore, at the time of bifurcation, the coefficient of $\frac{1}{2}p^{2}$, denoted by $a$, and the coefficient of $\frac{1}{6}q^{6}$, denoted by $b$, are
\begin{equation*}
a = -\frac{4}{5}, \quad 
b = -\frac{3}{20}.
\end{equation*}

We find the solutions $\nu_{1}(j,t) = 0$, given by
\begin{equation*}
j_{0}^{\pm}(t) = \frac{8t \pm \sqrt{(5t-1)(3t+1)}}{4t}.
\end{equation*}
Plugging this into the expression for $\nu_{2}(j,t)$ yields
\begin{equation*}
\nu_{2}(j_{0}^{\pm}(t),t) = \pm \frac{4t\sqrt{(5t-1)(3t+1)}}{1-t+\sqrt{(t-1)^{2}}}.
\end{equation*}
Since $\nu_{1}(j,t)$ vanishes for $t \geq \frac{1}{5}$, we need to restrict our evaluation to this region. Note that then the square root in the numerator is $\geq 0$, where we have equality for $t = \frac{1}{5}$. Furthermore, for $\frac{1}{5} \leq t \leq 1$ the denominator is $1-t+\abs{t-1} = 2(1-t)$. Thus $\nu_{2}(j_{0}^{+}(t),t) > 0$ and $\nu_{2}(j_{0}^{-}(t),t) < 0$ for $\frac{1}{5} < t < 1$. Thus, since $a < 0$, there is a
\begin{itemize}
    \item Hamiltonian flip bifurcation along $j_{0}^{-}(t)$, and
    \item dual Hamiltonian flip bifurcation along $j_{0}^{+}(t)$.
\end{itemize}

Finally, we find
\begin{equation*}
\frac{1}{2b} \frac{\left( \partial \nu_{2} / \partial j (\varepsilon,j_{0}(\varepsilon),t_{0}(\varepsilon)) \right)}{\partial^{2} \nu_{1} / \partial j^{2} (\varepsilon,j_{0}(\varepsilon),t_{0}(\varepsilon))} = \frac{4}{3}.
\end{equation*}
Thus, the bifurcation is non-degenerate with oppositely concave parabolas. The bifurcation diagram for certain values of $t$ has been drawn in Figure \ref{fig:W2-bif-diag}.

\begin{figure}[tb]
\def\scale{0.55}
\centering
\begin{tabular}{ccc} 
 {\includegraphics[scale=\scale]{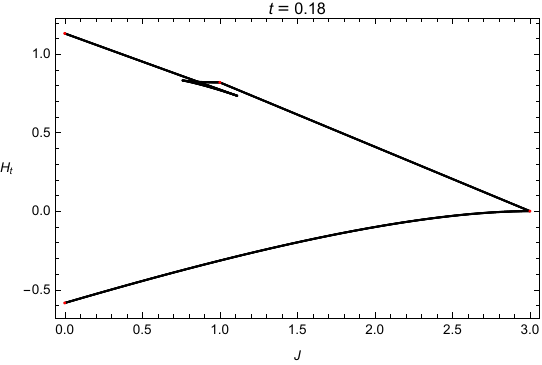}} & 
 {\includegraphics[scale=\scale]{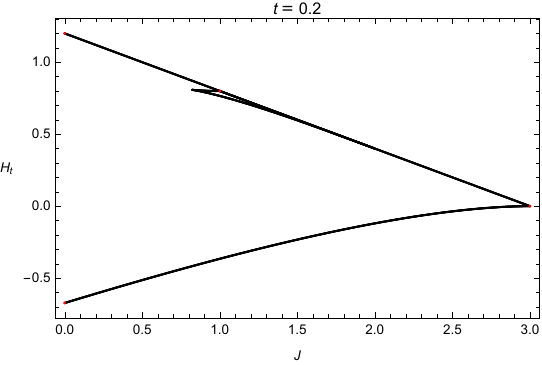}} & 
 {\includegraphics[scale=\scale]{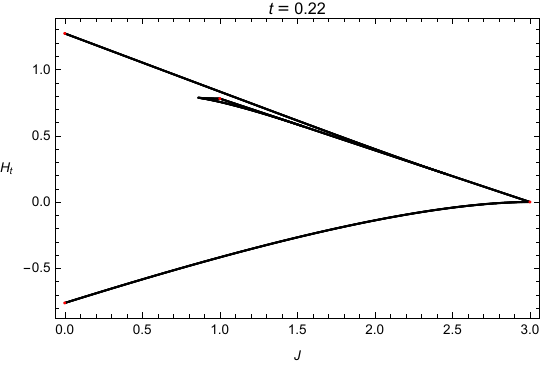}}
\end{tabular}
\caption{Bifurcation diagrams of the $W_{2}(1,1)$ Hirzebruch surface.}
\label{fig:W2-bif-diag}
\end{figure}
\printbibliography

@preamble{ " \newcommand{\noop}[1]{} " }

@book{Bolsinov2004,
  author    = {Bolsinov, A. V. and Fomenko, A. T.},
  title     = {Integrable {H}amiltonian {S}ystems: {G}eometry, {T}opology, {C}lassification},
  isbn      = {9780203643426},
  year      = {2004},
  publisher = {CRC Press}
}

@article{colindeverdiere2003,
  author = {Colin de Verdi{\`e}re, Y. and \vietnamese{{V{\~u} Ng{\d o}c}}, S.},
  title = {Singular {B}ohr-{S}ommerfeld Rules for {2D} Integrable Systems},
  url = {https://hal.science/hal-00020086},
  note = {postscript, 61 pages, figures best seen in color. Preprint Institut Fourier},
  journal = {{Annales Scientifiques de l'{\'E}cole Normale Sup{\'e}rieure}},
  publisher = {{Soci{\'e}t{\'e} math{\'e}matique de France}},
  volume = {36},
  pages = {1-55},
  year = {2003},
  HAL_ID = {hal-00020086},
  HAL_VERSION = {v1},
}

@inproceedings{Cushman1990,
title = "The Hamiltonian Hopf bifurcation in the Lagrange top",
author = "Cushman, R. H. and van der Meer, J. C.",
year = "1990",
doi = "10.1007/BFb0097463",
language = "English",
series = "Lecture Notes in Mathematics",
publisher = "Springer",
pages = "26--38",
editor = "C. Albert",
booktitle = "G{\'e}ometrie Symplectique et M{\'e}canique (La Grande Motte, France, May 1988)",
address = "Germany",
}

@article{Cushman2007,
  doi       = {10.1134/S156035470706007X},
  year      = 2007,
  month     = {12},
  volume    = {12},
  number    = {6},
  pages     = {642--663},
  author    = {Cushman, R. and Dullin, H. R. and Han{\ss}mann, H. and Schmidt, S.},
  title     = {The $1:\pm2$ resonance},
  journal   = {Regular and Chaotic Dynamics}
}

@article{Delzant1988,
	year = {1988},
	volume = {116},
	number = {3},
	pages = {315-339},
	author = {Delzant, T.},
	title = {Hamiltoniens p{\'e}riodiques et images convexes de l'application moment},
	journal = {Bulletin de la Soci{\'e}t{\'e} math{\'e}matique de France}
}

@article{DullinPelayo2016,
  author  = {Dullin, H. R. and Pelayo, {\'A}.},
  title   = {Generating Hyperbolic Singularities in Semitoric Systems Via {H}opf Bifurcations},
  journal = {Journal of Nonlinear Science},
  year    = {2016},
  month   = {6},
  day     = {01},
  volume  = {26},
  number  = {3},
  doi     = {10.1007/s00332-016-9290-0},
  pages   = {787--811}
}

@article{Efstathiou2004,
 author = {Efstathiou, K. and Cushman, R. H. and Sadovskií, D. A.},
 doi = {10.1016/j.physd.2004.03.003},
 journal = {Physica D},
 journalurl = {https://www.elsevier.com/locate/physd},
 number = {3-4},
 pages = {250-274},
 title = {Hamiltonian Hopf bifurcation of the hydrogen atom in crossed fields},
 volume = {194},
 year = {2004}
}

@article{Efstathiou2007,
  author  = {Efstathiou, K. and Cushman, R.H. and Sadovski{\'i}, D.A.},
  journal = {Advances in Mathematics},
  number  = {1},
  pages   = {241--273},
  title   = {Fractional Monodromy  in the 1:-2 resonance},
  volume  = {209},
  year    = {2007}
}

@misc{Efstathiou2024,
      title={On the affine invariant of hypersemitoric systems}, 
      author={Efstathiou, K. and Hohloch, S. and Santos, P.},
      year={2024},
      eprint={2411.17509},
      archivePrefix={arXiv},
      primaryClass={math.SG},
}

@mastersthesis{Flamand2024,
    author = {Flamand, N.},
    title = {Hirzebruch surfaces, flaps and swallowtails},
    school = {University of Antwerp},
    year = {2024}
}

@book{Gibson1979,
  title={Singular Points of Smooth Mappings},
  author={Gibson, C. G.},
  isbn={9780273084105},
  year={1979},
  publisher={Pitman}
}

@article{Hanssmann2002,
title = "On the Hamiltonian Hopf bifurcations in the 3D H{\'e}non-Heiles family",
author = "Han{\ss}mann, H. and van der Meer, J. C.",
year = "2002",
doi = "10.1023/A:1016343317119",
language = "English",
volume = "14",
pages = "675--695",
journal = "Journal of Dynamics and Differential Equations",
issn = "1040-7294",
publisher = "Springer",
number = "3",
}

@book{Hanssmann2006,
  author    = {Han{\ss}mann, H.},
  title     = {Local and {S}emi-{L}ocal {B}ifurcations in {H}amiltonian {D}ynamical {S}ystems},
  publisher = {Lecture Notes in Mathematics, Springer, Berlin, Heidelberg},
  year      = {2007}
}

@article{Henriksen2025,
  doi = {10.1088/1361-6544/adf9b0},
  year = {2025},
  month = {8},
  publisher = {IOP Publishing},
  volume = {38},
  number = {8},
  pages = {085022},
  author = {Henriksen, T. V.},
  title = {Hamiltonian Hopf bifurcations in Gaudin models},
  journal = {Nonlinearity},
}

@article{Kalashnikov1998,
doi = {10.1070/IM1998v062n02ABEH000173},
year = {1998},
month = {4},
publisher = {},
volume = {62},
number = {2},
pages = {261},
author = {Kalashnikov, V. V.},
title = {Typical integrable {H}amiltonian systems on a four-dimensional symplectic manifold},
journal = {Izvestiya: Mathematics}
}

@article{LeFloch2018,
  shorthand = {LeFP},
  doi = {10.48550/ARXIV.1810.06915},
  author = {Floch, Y. Le and Palmer, J.},
  title = {Semitoric families},
  journal = {Memoirs of the AMS},
  year = {2018},
  copyright = {arXiv.org perpetual, non-exclusive license}
}

@article{Martynchuk2017,
  author   = {Martynchuk, N. and Efstathiou, K.},
  title    = {Parallel Transport Along {S}eifert Manifolds and Fractional Monodromy},
  journal  = {Communications in Mathematical Physics},
  year     = {2017},
  month    = {12},
  day      = {01},
  volume   = {356},
  number   = {2},
  pages    = {427--449},
  issn     = {1432-0916},
  doi      = {10.1007/s00220-017-2988-5}
}

@article{Mather1968,
  author = {Mather, J. N.},
  journal = {Publications Mathématiques de l'IHÉS},
  language = {eng},
  pages = {127-156},
  publisher = {Institut des Hautes Études Scientifiques},
  title = {Stability of $C^\infty $ mappings, III. Finitely determined map-germs},
  volume = {35},
  year = {1968},
}

@article{Nekhoroshev2002,
  author  = {Nekhoroshev, N. N. and Sadovski{\'i}, D.A. and Zhilinski\'{\i}, B.I.},
  journal = {Comptes Rendus. Mathématique},
  pages   = {985-988},
  title   = {Fractional monodromy of resonant classical and quantum oscillators},
  number  = {11},
  volume  = {335},
  year    = {2002},
  doi     = {10.1016/S1631-073X(02)02584-0}
}

@article{Nekhoroshev2006,
  author  = {Nekhoroshev, N. N. and Sadovski{\'i}, D.A. and Zhilinski\'{\i}, B.I.},
  journal = {Annales Henri Poincar\'{e}},
  pages   = {1099-1211},
  title   = {Fractional {H}amiltonian Monodromy},
  volume  = 7,
  year    = 2006
}

@article{Pelayo2009,
author = {Pelayo, A. and V{\~u} Ng{\d o}c, S.},
title = {Semitoric integrable systems on symplectic 4-manifolds},
volume = {177},
journal = {Invent. math.},
pages = {571--597},
year = {2009},
doi = {10.1007/s00222-009-0190-x},
URL = {https://doi.org/10.1007/s00222-009-0190-x}
}

@article{Pelayo2011,
author = {Pelayo, A. and V{\~u} Ng{\d o}c, S.},
title = {{Constructing integrable systems of semitoric type}},
volume = {206},
journal = {Acta Mathematica},
number = {1},
publisher = {Institut Mittag-Leffler},
pages = {93--125},
year = {2011},
doi = {10.1007/s11511-011-0060-4},
URL = {https://doi.org/10.1007/s11511-011-0060-4}
}

@book{vanderMeer1985,
  address   = {Springer-Verlag Berlin Heidelberg},
  author    = {van der Meer, J. C.},
  booktitle = {Lecture Notes in Mathematics},
  doi       = {10.1007/BFb0080357},
  isbn      = {978-3-540-16037-3},
  pages     = {VI, 115},
  publisher = {Springer},
  title     = {The {H}amiltonian {H}opf {B}ifurcation},
  year      = {1985}
}

@book{vanderMeer2017,
title = "On the geometry of {H}amiltonian systems : lecture notes seminar {GISDA} - {U}niversidad del {B}io {B}io",
author = "van der Meer, J. C.",
year = "2017",
month = jun,
language = "English",
series = "CASA Reports",
publisher = "Technische Universiteit Eindhoven",
}

\end{document}